\documentclass[10pt]{siamltex}
\usepackage{array}
\usepackage{amssymb}
\usepackage{bm}
\usepackage{amsmath}
\usepackage{comment}
\usepackage{amsfonts}
\usepackage{epstopdf}
\usepackage{algorithm}
\usepackage{textcomp}
\usepackage{hyperref}
\usepackage{tikz}
\usepackage{graphicx,color} %
\usepackage{subfigure}

\graphicspath{{img/all/}}

\ifpdf
  \DeclareGraphicsExtensions{.eps,.pdf,.png,.jpg}
\else
  \DeclareGraphicsExtensions{.eps}
\fi

\newcommand{\TODO}[1][]{{\color{red}TODO\ifthenelse{\equal{#1}{}}{}{: #1}}}

  \usepackage{todonotes}

\newcommand{\removed}[1] {\ifmmode{\color{red}\cancel{#1}}\else{\color{red}\sout{#1}}\fi}

\newcommand{\e}{\mathrm{e}}
\newcommand{\scurl}[1][]{\textup{curl}\ifthenelse{\equal{#1}{}}{}{_#1}}
\newcommand{\vcurl}[1][]{\textup{\textbf{curl}}\ifthenelse{\equal{#1}{}}{}{_#1}}
\newcommand{\jump}[1]{\left[\hspace{-0.025in}\left[#1\right]\hspace{-0.025in}\right]}
\newcommand{\sdiv}[1][]{\textup{div}\ifthenelse{\equal{#1}{}}{}{_#1}}
\newcommand{\vdiv}[1][]{\textup{\textbf{div}}\ifthenelse{\equal{#1}{}}{}{_#1}}

\newcommand{\dx}[1][x]{\,\mathrm{d}#1}

\def\cT{\mathcal{T}}
\newcommand{\norm}[1]{\| #1 \|}

\newcommand{\Tnorm}[1]{|||#1|||}

\newcommand{\seminorm}[1]{| #1 |}

\begin{document}

\title{Implicit-explicit Crank-Nicolson scheme for Oseen's equation at high Reynolds number}

\author{Erik Burman\thanks{Department of Mathematics, University College London, London, UK--WC1E 6BT, UK, 
 (e.burman@ucl.ac.uk, d.garg@ucl.ac.uk). } 
\and Deepika Garg\footnotemark[1]
\and  Johnny Guzman\thanks{Division of Applied Mathematics, Brown University, Providence, RI, USA,
(johnny\_guzman@brown.edu). }
}



\maketitle

\begin{abstract}
In this paper we continue the work on implicit-explicit (IMEX) time discretizations for the incompressible Oseen equations that we started in \cite{BGG23} (E. Burman, D. Garg, J. Guzm\`an, {\emph{Implicit-explicit time discretization for Oseen's equation at high Reynolds number with application to fractional step methods}}, SIAM J. Numer. Anal., 61,  2859--2886, 2023). The pressure velocity coupling and the viscous terms are treated implicitly, while the convection term is treated explicitly using extrapolation. Herein we focus on the implicit-explicit Crank-Nicolson method for time discretization. For the discretization in space we consider finite element methods with stabilization on the gradient jumps. The stabilizing terms ensures inf-sup stability for equal order interpolation and robustness at high Reynolds number. Under suitable Courant conditions we prove stability of the implicit-explicit Crank-Nicolson scheme in this regime. The stabilization allows us to prove error estimates of order $O(h^{k+\frac12} + \tau^2)$. Here $h$ is the mesh parameter, $k$ the polynomial order and $\tau$ the time step. Finally we discuss some fractional step methods that are implied by the IMEX scheme. Numerical examples are reported comparing the different methods when applied to the Navier-Stokes' equations.
\end{abstract}



\section{Introduction}
As a model problem for Navier-Stokes' equations we will consider the transient Oseen's equation.
Let $\Omega \subset \mathbb{R}^d$, $d=2,3$ be an open polygonal
domain. Let $I:= (0,T)$ and denote the space time domain by $\mathcal{M} =
\Omega \times I$. 
Now, consider the functional spaces ${V}=\{{v} \in [{H}^{1}_{0}(\Omega)]^{d}\}$, and $Q = {L}^{2}_{0}(\Omega)= \left\{q \in {L}^{2}(\Omega) | \ \int_{\Omega}{q}\dx = 0 \  \right\}$.
Assuming sufficient smoothness of the solution, the fluid equations that we consider may be written,
\begin{alignat}{2}\label{eq:convdiff}
\partial_t u + \beta \cdot \nabla u + \nabla p- \mu \Delta u = & f \quad && \mbox{ in }
                                                        \mathcal{M}, \\
\nabla \cdot u = & 0 \quad && \mbox{ in }
                                                        \mathcal{M}, \\
u(\cdot, 0) =& u_0 \quad && \mbox{ in }
                                                        \Omega, \\
u =& 0 \quad && \mbox{ on } \partial \Omega \times I.
\end{alignat}
Here $f \in [L^2(\Omega)]^d$, $u_0 \in V$, with $\nabla \cdot
u_0 = 0$, $\beta \in [W^{1,\infty}(\mathcal{M})]^d$, with $\nabla \cdot \beta=0$, $\beta \cdot n|_{\partial \Omega}=0$. Let $\beta_{\infty}=\|\beta\|_{[L^\infty(\mathcal{M})]^d}$, where $\mathcal{M}$ is a space time domain and $\beta$ is the flow velocity, since we are interested in high Reynolds number $\beta_\infty>0$. We define $\beta_{1,\infty} = \|\beta\|_{[W^{1,\infty}(\mathcal{M})]^d}$. This is a parabolic
problem and it is known to admit a unique solution $u$ in
$[L^2(0,T;H^1_0(\Omega)) \cap L^\infty(0,T;L^2(\Omega))]^d$. We wish to discretize this system using stabilized finite elements in space and an implicit-explicit (IMEX) time discretization. In such a scheme the viscous terms together with the velocity pressure coupling are treated implicitly and the convection terms and associated stabilization terms are treated explicitly, with extrapolation. This has the advantage that no nonlinear system has to be solved and the velocity stabilization does not change the Galerkin stencil. Indeed for each time step a standard Stokes' problem has to be solved, for which there exists, efficient preconditioners \cite{ESW}.
The application of IMEX scheme's for transient Navier-Stokes' equations in the high Reynolds regime appears to be very popular in applications. There are typically two approaches available in the literature. Either second order backward differentiation (BDF2) is used for the time discretization \cite{MCBS22,FWK18}, or a Runge-Kutta type approach is used \cite{LMS23, CCRL17}. The properties of time discretization schemes are known to depend strongly on the mesh Reynolds number $Re:=\frac{h {\beta_{\infty}}  }{\mu}$. Here $h$ the mesh size, and $\mu$ the viscosity parameter.
There seems to be few works in the literature where IMEX scheme's for Oseen's equation are analysed in the case where $Re$ is high,
In our previous work \cite{BGG23} we proved stability and error estimates for the IMEX BDF2 method in the high $Re$ limit, under standard hyperbolic CFL for affine approximation or a slightly stronger so-called 4/3-CFL condition $(\tau \leq Co_{4/3} h^{\frac43})$ if $\tau$ is the time discretization parameter) for polynomial orders greater than one. Herein, we will consider the IMEX Crank-Nicolson scheme for time discretization (i.e. using Adams-Bashforth 2 for the convective term) and prove similar results in this case.  

For the Oseen's equations, the standard Galerkin finite element space discretization is known to have poor stability properties in the high Reynolds regime and therefore stabilized methods need to be considered. It is known that stabilizations such as Galerkin-Least squares, or Streamline Upwind Petrov Galerkin \cite{FF92} couples to the time discretization in a nontrivial way, prompting space-time finite element solutions. So called symmetric stabilizations \cite{BB06,BFH006, Cod08} on the other hand typically extend the matrix stencil, making the computational solution of the time steps more costly. In this paper we consider the symmetric stabilization where a penalty is applied to the jumps of the gradients over element faces, known as Continuous Interior Penalty \cite{DD76, BH04}, however it is straightforward to extend most of the results to other symmetric stabilization methods or discontinuous Galerkin methods.

There exists a vast literature on IMEX schemes for the Navier-Stokes' equations with a wide range of different stability conditions (see the introduction of \cite{he2007stability} for an overview). We will focus on the Crank-Nicolson case considered in this paper here and refer to \cite{BGG23} for a more general overview. The IMEX Crank-Nicolson scheme was first considered by Kim and Moin in a fractional step method \cite{kim1985application} it was analysed by Marion and Temam \cite{marion1998navier} and then by several authors, see for example \cite{he2007stability,tone2004error}. Recently the approach has been extended to other fluid models such as $\alpha$-models \cite{MRTT16}, deconvolution methods \cite{KMR14} and a Burgers' type model problem for compressible flow \cite{ZJH18}.
In all these works the viscosity is assumed to be strictly positive, and the mesh Reynolds number small, with the admissible time step going to zero as the viscosity goes to zero. To the best of our knowledge there is no numerical analysis of the IMEX Crank-Nicolson scheme available in the literature with results that are robust for arbitrarily high Reynolds number.

Herein we will only consider stability and error analysis for the case when the mesh Reynolds number is high, $Re\ge 1$. For low Reynolds number flow discretized using IMEX Crank-Nicolson we refer to the papers cited above. Note that the analysis for BDF2 with extrapolated convection does not carry over in to the present case, since the BDF2 scheme has some built in dissipation that is useful to control the extrapolation error. Since no such explicit dissipative mechanism is available for the Crank-Nicolson scheme, substantial effort in the present work goes into proving the corresponding a priori bounds through other means (Lemma \ref{lem:time_diss} and Lemma \ref{lemmaY_1}). In this context, the case of linearized incompressible flow includes several challenges not present in the scalar case, handled in \cite{BG22}, related to the pressure velocity coupling and associated stabilization. Some auxiliary results carry over from our previous work under only minor modifications and for those cases we only state the results without proof below.

\subsection{Outline of the paper} In the next section we introduce the weak formulation of the model problem and the semi-discretization in space. Then in Section \ref{sec:full_disc} we introduce the Crank-Nicolson scheme with extrapolated convection (Adams--Bashforth 2) and collect some technical results. In Section \ref{sec:stab_error} we propose a stability and error analysis for the given scheme in the high Reynolds regime. In the case of inviscid flow problems we then recall from \cite[Section 6]{BGG23}, in Section \ref{sec:split}, that the implicit-explicit method naturally can be written as a split method based on Poisson pressure projection steps, without any splitting error and give the formulation in the Crank-Nicolson case. This splitting scheme is similar to the inviscid version of the split IMEX Crank-Nicolson method introduced in  \cite{johnston2004accurate} for viscous flow. We then discuss how to make a consistent splitting scheme for all Reynolds numbers. The stability of this generalization is explored numerically. Finally in Section \ref{sec:numeric} we report some numerical results or the imex method and the splitting scheme. Some additional numerical experiments are provided in the supplementary material.
\section{Weak formulation and space semi-discretization}\label{sec:space_semi_disc}
We will consider the standard notation for Sobolev spaces ${{H}}^{ m}(\Omega)$, with norm $\norm{\cdot}_{ H^{m}(\Omega)}$, $m \geq 0$.  The scalar product in ${L}^{2}(\Omega)$ is denoted by $(\cdot,\cdot)_{\Omega}$ and its norm by $ \norm{u}$.  The notations $[L^2(\Omega)]^{d}$ and $[{ H}^{1}(\Omega)]^{d}$ respectively abbreviate the vector-valued versions of $L^2(\Omega)$ and ${ H}^{1}(\Omega)$; ${ H}_{0}^{1}(\Omega)$ is a subspace of ${H}^{1}(\Omega)$ with zero trace functions.
We define the forms, for $X \subset \mathbb{R}^d$,
\[
(u,v)_X :=\int_X u \cdot v ~ \mbox{d}x, \quad
F(v):=\int_\Omega f \cdot v ~\mbox{d}x,
\]
\[
c(w,v) :=\int_\Omega (\beta \cdot \nabla) w \cdot v ~\mbox{d}x,\quad a(w,v)
:= \int_\Omega \mu \nabla w : 
\nabla v~\mbox{d}x,
\]
and
$
b(q,w) := -\int_\Omega q \nabla \cdot w ~\mbox{d}x.
$
   The system \eqref{eq:convdiff} may then be cast on the weak form,
\begin{subequations}
\begin{alignat}{2}\label{eq111}
(\partial_t u, v)_\Omega + c(u,v)  + b(p,v) + a(u,v) =& F(v), \quad && \forall v \in
V, \, t>0,  \\
b(q,u) = & 0  , \quad && \forall q \in
Q, \, t>0. \label{eq113}
\end{alignat}
\end{subequations}
We will use the following two norms $\norm{\cdot}^2=(\cdot, \cdot)_{\Omega}$ and $\norm{v}_{\infty}= \sup_{x \in \bar \Omega}|v(x)|$.
Let $\{\mathcal{T}_h\}_h$ be a family of affine, simplicial, quasi uniform meshes of $\Omega$. The mesh parameter is defined by $h:= \max_{T \in \mathcal{T}_h} \mbox{diam}(T)$ and we assume $h<1$. We assume that the meshes are kept
fixed in time and, for simplicity, the family $\{\mathcal{T}_h\}_h$ is supposed to be quasi-uniform. Mesh faces
are collected in the set $\mathcal{F}_h$ which is split into the set of interior faces, $\mathcal{F}^{int}_h$, and of boundary
faces, $\mathcal{F}^{ext}_h$. For a smooth enough function $v$ that is possibly double-valued at $F \in \mathcal{F}^{int}_h$ with $F=\partial{T}^{-} \cap \partial{T}^{+}$, we define its jump at $F$ as $\jump{v}=:v_{T^{-}}-v_{T^{+}}$, and we fix the unit normal vector to $F$, denoted by $\nu_F$, as pointing from $T^-$ to $T^+$. The arbitrariness in the sign of  $\jump{v}$ is irrelevant in what follows.

We consider continuous finite elements with equal-order to discretize in space the velocity
and the pressure. Let $k \geq 1$ be an integer.
We will also make use of the piece-wise polynomial space	
\[
X_h^k := \left\{v \in {L}^{2}(\Omega) \ : \ v|_T \in \mathcal{P}_k(T) ~~ \forall ~ T \in \cT_{h}   \right\}.
\]
and set
\[
W_h := X_h^k \cap C^0(\bar \Omega),
\]
with $\mathcal{P}_k(T)$ spanned by the restriction to $T$ of polynomials of total degree $\leq$ $k$. Set \[ V_h=[W_h]^d , \ \ Q_h= W_h \cap L^{2}_0(\Omega).\]
Let $\pi_h$ the $L^2$-orthogonal projection onto $W_h$ given by
$
    (\pi_h w, w_h)_{\Omega}=(w, w_h)_{\Omega}$, $\forall w_h \in W_h.   
$

We let $\pi_0:L^2(\Omega) \rightarrow X_h^0$ be the $L^2$-orthogonal projection:
$
    (\pi_0 w, v_h)_{\Omega}=(w, v_h)_{\Omega},$ $\forall v_h \in X_h^0.
$

The space semi-discretized scheme we propose reads on abstract form:
For all $t \in (0,T)$, find $(u_h(t),p_h(t)) \in V_h \times Q_h$ such that
\begin{align}\label{eq:scheme_semi}
\partial_t u_h + C_h u_h+A_hu_h+G_h p_h 
&= F_h, \nonumber\\
G_h^*u_h+S^p_h p_h &=0, 
\end{align}
{where $F_h \in V_h$ satisfies $(F_h, v_h)_\Omega=F(v_h), \forall v_h \in V_h$.}
The
 discrete operators introduced in (\ref{eq:scheme_semi}) are defined as follows.
For $w \in X_h^k + [H^{\frac32+\epsilon}(\mathcal{T})]^d$ we define $C_h(t) w \in V_h$ by
\begin{alignat}{2}\label{eq:Cdef}
(C_h(t) w,v_h )_{\Omega} :=& \sum_{T \in \mathcal{T}_h} (\beta(t) \cdot \nabla w, v_h)_{T}+ s_u(w,v_h)  \quad & \forall v_h \in V_h, 
\end{alignat}
where $s_u(\cdot, \cdot )$ is the stabilization operator defined as (see for example \cite{BH04,BFH006, BE07})
\begin{equation}\label{eq:stab_form}
s_u(w_h,v_h) := \gamma_u\sum_{F \in \mathcal{F}^{int}_h} \int_F h_F^2 (|\beta \cdot n|+\beta_\infty\varepsilon^\perp )
\jump{\nabla w_h} : \jump{\nabla v_h} ~\mbox{d}s+ (\beta_\infty w_h\cdot n, v_h \cdot n)_{\partial
  \Omega}, 
\end{equation}
with $\gamma_u \ge 0$, $\varepsilon^\perp>0$. The presence of $\varepsilon^\perp$ avoids the need of any separate grad-div stabilization term thanks to the added crosswind stabilization. { In practice it is typically chosen small (in the computations below $\varepsilon^\perp =O(10^{-2}|\beta|)$.} We define the semi-norm $|v|_{s_u}^2:= s_u(v, v)$ and note that using integration by parts that
\begin{equation}\label{Chsemi}
(C_h(t) v_h,v_h )_{\Omega}=  |v_h|_{s_u}^2 \quad \forall v_h \in V_h.
\end{equation}
To be precise the semi-norm $|\cdot |_{s_u}$ depends on time $t$ through $\beta$.  
We will use $C_h^n$ to denote $C_h(t^n)$. 
For $w \in X_h^k + [H^{\frac32+\epsilon}(\mathcal{T})]^d$ we define $A_h w \in V_h$ by
\begin{multline*}
    (A_h w,v_h)
    := a(w,v_h)- (\mu n \cdot \nabla  w,v_h)_{\partial \Omega}-(\mu n \cdot \nabla v_h , w)_{\partial
  \Omega} + \frac{\gamma \mu}{h}( w,v_h)_{\partial
  \Omega}.
\end{multline*}

For $q \in Q_h + H^{1}(\Omega)$ we define $G_h q \in V_h$ by
\[
(G_h q,v_h)_{\Omega}
:= - (q,\nabla \cdot v_h)_{\Omega} + (q,v_h \cdot n)_{\partial
  \Omega}  \quad \forall v_h \in V_h,
\]
similarly we define the adjoint $G^*_h w \in W_h$
\[
(G^*_h w,q_h)_{\Omega} :=  (q_h,\nabla \cdot w)_{\Omega} - (q_h,w \cdot n)_{\partial
  \Omega} \quad \forall q_h \in W_h.
\]
The following relationship easily follows
\begin{equation}\label{Gh}
(G^*_h v_h,q_h)_{\Omega} = -(v_h, G_h q_h)_{\Omega} \quad  \forall v_h \in V_h, \ q_h \in W_h. 
\end{equation}

For $q \in Q_h + H^{\frac32+\epsilon}(\mathcal{T})$, 
Let $
(S^p_h q,q_h)_{\Omega} = s_p(q,q_h), \forall q_h \in Q_h,
$
where
\begin{align} \label{pres_stab}
s_p(q,q_h) := \frac{\gamma_p\xi h^3}{\mu} \sum_{F \in \mathcal{F}^{int}_h} \int_F 
\jump{\nabla q} \cdot \jump{\nabla q_h} ~\mbox{d}s,
\end{align}
$\xi = \min\{1,Re^{-1}\}$.
We define the semi-norm $|q_h|^{2}_{s_p}:= (S^p_h q_h,q_h)_{\Omega}$. We define the energy norm as follows:
\begin{equation}\label{eq:energy}
\|v_h\|_E^2:= |v_h|_{s_u}^2+ \|v_h\|_{A}^2, \text{ where }  \|v_h\|_{A}^2 := (A_h v_h, v_h)_{\Omega}.
\end{equation}
This is indeed a norm on $V_h$ if $\gamma$ is sufficiently large.  To be precise the energy norm depends on time $t$ since $|\cdot |_{s_u}$ depends on the $t$ through $\beta$.  

For a linear operator $L$ defined for $V_h$ we define
\begin{equation*}
    \|L\|_{h} := \sup_{v_h \in V_h}\frac{L(v_h)}{(\|v_h\|_E^2 +\|v_h\|^2)^{\frac{1}{2}}}.
\end{equation*}
For a linear operator $\tilde{L}:Q_h \mapsto \mathbb{R}$, such that $\tilde{L}(q_h) = 0$ for all $q_h \in \mbox{ker} S^p_h$ we define, with $Q_{h,S} := Q_h \setminus \mbox{ker} S^p_h$
\begin{equation}\label{eq:L_div_cont}
    \|\tilde{L}\|_{\mathrm{div}} := \sup_{q_h \in Q_{h,S}}\frac{\tilde{L}(q_h)}{|q_h|_{s_p}}.
\end{equation}

Assuming sufficient regularity of the exact solution the above formulation is
strongly consistent. More generally we have the following result.
\begin{lemma} \label{consistency} (\em modified Galerkin orthogonality). 
Assume that $(u, p)$, the solution
of (\ref{eq:convdiff}), belongs to the space $[H^{\frac{3}{2}+\epsilon}(\Omega)]^{d+1}
$, with $\epsilon > 0$, and let $(u_h, p_h) \in V_h \times Q_h$ be the solution of (\ref{eq:scheme_semi}). Then
\begin{align*}
(\partial_t (u-u_h),v_h) &+ (C_h (u-u_h),v_h)+(A_h (u-u_h),v_h)
\\ &+(G_h (p-p_h), v_h )+(G_h^*(u-u_h), q_h)+(S^p_h (p-p_h), q_h) =0
\end{align*}
for all $(v_h, q_h) \in V_h \times Q_h$.
\end{lemma}
\begin{proof}
This is an immediate consequence of the consistency of the standard Galerkin method with Nitsche weak imposition of the boundary conditions and the fact that, under the regularity assumptions, $s_u(u,v_h)=0$ and $s_p(p,q_h)=0$, since $\jump{\nabla u}_{F}=0$ and $\jump{\nabla p}_{F}=0$ for all interior faces $F$.
\end{proof}

\section{Fully discrete methods}\label{sec:full_disc}
Here we present the Crank-Nicolson time discretizations based on implicit treatment of the Stokes part and explicit treatment of the convection term, with extrapolation. 
We consider general right hand sides for both the momentum and the mass equations. This is helpful for the error analysis.  
\subsection*{Crank-Nicolson with extrapolated convection}
We let $\tau$ be the time step size. We define 
 the increment operator $\delta$ such that 
 \begin{align}\label{delta}
\delta v^{n+1}:= v^{n+1} - v^n.
\end{align}
We may then write the extrapolated
Crank-Nicolson scheme as follows.
Given $u^0$ and $u^1$, for $n=1, 2,\hdots, N-1$, find $(u_h^{n+1}, p_h^{n+1}) \in V_h \times Q_h$ such that
\begin{align}\label{eq:momentum_CN}
\tau^{-1} \delta u_h^{n+1} + G_h p_h^{n+1} +  C_h^{n+\frac{1}{2}} \hat u_h^{n+1}+ A_h \bar u_h^{n+1} &= {L_h^{n+1}}, \\G^*_h  u_h^{n+1} + S^p_h p_h^{n+1} &= \tilde{L}_h^{n+1}.\label{eq:mass_CN}
\end{align}
where $C^{n+\frac{1}{2}}_h=C_h(t^{n+\frac{1}{2}})$, $u_h^0 = \pi_h u^0$, $u_h^1 = \pi_h u^1$, $\hat u_h^{n+1} = \frac32 u_h^n -\frac12 u_h^{n-1}$ and $\bar u_h^{n+1}=
(u_h^{n+1}+u_h^{n})/2$. Note that the quantities approximate the exact solution at different time levels, indeed $u_h^{n} \approx u(t^{n})$, $\hat u_h^{n} \approx u(t^{n-\frac12})$, $p_h^n \approx p(t^{n-\frac12})$.
${L}_h^n$ and $\tilde{L}_h^n$ are defined by   
\begin{align*}
  ({L}_h^n, v_h)&={L}^n(v_h) \ \forall v_h \in v_h \mbox{ and } (\tilde{L}_h^n, q_h)=\tilde{L}^n(q_h) \ \forall q_h \in Q_h,
\end{align*} where $\{\tilde{L}^{n}\}$ and $\{\tilde{L}^{n}\}$ are bounded linear operators on $V_h$ and $Q_h$ respectively.
Also note that since $u_h^1$ is given we may define
$( G_h^*  u_h^{1},  q_h)_{\Omega}=\tilde{L}^1(q_h).$

In case the schemes are used for the approximation of \eqref{eq111}--\eqref{eq113} we take $L_h = F_h$ and $\tilde L = 0$.
In this case the scheme needs approximate velocities at two previous time levels, so that the first approximation will be of $u(t^{2})$. To simplify the exposition we assume that $u(t^1)$ with $\nabla \cdot u_1 = 0$ is known. In the Crank-Nicolson method the approximation of the time derivative is given by the scaled increment operator $\tau^{-1}\delta u^{n+1}$. The extrapolation is taken to the
time level $t^{n+\frac{1}{2}}$, in order to approximate the central difference in time that is the key feature of the Crank-Nicolson scheme.

\subsection{Technical Results} We will use the notation $C$ for a generic constant which may be different at different occasions, but is always independent of $h$ and $\mu$. For the analysis, we need  the following technical results. First we introduce the standard inverse and trace inequalities.
The family $\cT_h$ satisfies the following trace and inverse inequality
\begin{alignat}{2}\label{trace_h1}
    \norm{v}_{L^2(\partial T)} \leq & C(h^{-\frac{1}{2}}\norm{v}_{L^{2}(T)}+h^{\frac{1}{2}}\norm{\nabla v}_{L^{2}(T)}), \quad && v \in H^{1}(T),\\
    \label{inverse}
    \norm{\nabla v}_{L^{2}(T)} \leq & C h^{-1} \norm{v}_{L^{2}(T)}, \quad && v \in\mathcal{P}_k(T),\\
    \label{tracePk}
    \norm{v}_{L^2(\partial T)} \leq &  Ch^{-\frac{1}{2}}\norm{v}_{L^{2}(T)}, \quad && v \in \mathcal{P}_k(T),
\end{alignat}
where the constant $C$ is independent of $T \in \cT_h \ \text{and} \ h$.
One immediate consequence is the following estimate:
\begin{equation}\label{invA}
\|v_h\|_A \le C(1+\sqrt{\gamma}) \frac{\sqrt{\mu}}{{h}} \|v_h\| \quad \forall v_h \in V_h. 
\end{equation}

Also, using the above inverse estimates one has:
  \begin{align}\label{stab_bound}
     |v_h|_{s_u}^2 \leq  C \frac{\beta_\infty}{h} \norm{v_h}^2 \quad \forall v_h \in V_h.
  \end{align}

We also observe that
\begin{equation}\label{diff_stab}
\|\delta^m v^n\| \leq 2 \sum_{i=0}^{m} \|v^{n-i}\|,
\quad m=1,2, \ \text{and} \ 2 \leq n \leq N.
\end{equation}

We define
 $h_\beta := \frac{h}{\beta_{\infty}}$. We will also introduce
a weighted $L^2$-norm,
$
\|v\|_\beta := \|h_\beta^{\frac12} v\|_\Omega.
$
We recall a discrete interpolation property (see \cite[Lemma 2.1]{BG22} for a similar result):  for $w_h \in V_h$ there holds that
\begin{equation}\label{interpinq}
 \inf_{v_h \in  V_h} \| \beta \cdot \nabla w_h -  v_h\|_\beta \leq C( s_u(w_h,w_h)^{\frac12} + \beta_{1,\infty} h_{{\beta}}^{\frac12} \|w_h\|).
\end{equation}

For the divergence and the pressure similar discrete interpolation estimate hold, for the proof we refer to \cite[Corollary 3.2]{BFH006}; for $w_h \in V_h$ there holds that
\begin{equation}\label{interdiv}
 \inf_{v_h \in W_h}\|{  \nabla\cdot  {w}_h }- v_h \| \leq
 C (\varepsilon^\perp \beta_\infty h)^{-\frac12} s_u( w_h, w_h)^{\frac12},
\end{equation}
for $q_h \in Q_h$ there holds that
\begin{align*} 
    \inf_{v_h \in V_h}\| \nabla q_h - v_h\| \leq C
 \left(\frac{\mu}{ \xi h^2}\right)^{\frac12}  s_p(q_h,q_h)^{\frac12}. 
\end{align*}
In the large Reynolds regime, $Re>1$, it follows from the definition of $\xi$ that
\begin{align} \label{pres_inter}
   \inf_{v_h \in V_h}\| \nabla q_h - v_h\| \leq C
 h_\beta^{-\frac12} s_p(q_h,q_h)^{\frac12}.
\end{align}
We will need the following Courant number's 
\begin{equation*}
Co :=   (\beta_{\infty}+1) \frac{\tau}{h}, \quad
Co_{4/3} :=  \frac{\tau}{h_\beta^{\frac43}}.
\end{equation*}
Observe that $Co$ is a free parameter that can be made
as small as we like by making $\tau$ small relative to $h$. 

The following estimates are shown in the proof of \cite[Lemma 2.2]{BG22}.
\begin{lemma}
It holds,
\begin{alignat}{1} 
  \|C_h v_h\| \leq & C h_{\beta}^{-1} \|v_h\|,  {\mbox{ and hence } \tau \|C_h v_h\| \leq  C Co \|v_h\|}\label{estimate_Ch}, \\
  \tau \|C_h v_h\| \leq & C \tau^{1/4} Co_{4/3}^{\frac{3}{4}} \|v_h\|  \label{estimate_Ch2}. 
\end{alignat}
\end{lemma}

Also, an immediate consequence of \eqref{invA} and \eqref{stab_bound} is the followoing. 
\begin{lemma} \label{enery_est_222}
It holds,
\begin{align} \label{enery_est}
\tau \|v_h\|_{E}^2 \leq  C_E\norm{v_h}^2 \ \text{for all} \ v_h \in V_h, 
\end{align}
where $C_E= C Co(\frac{1+\gamma}{Re}+1)$.
\end{lemma}

We will also need the following summation by parts formula. 
\begin{lemma}\label{lem:sum_by_parts_spec}
There holds that
\begin{alignat*}{1}
\sum_{n=1}^{N-1}  (C_h^{n+\frac{1}{2}} \delta w_h^{n+1}, v_h^{n+1})_{\Omega} = & - \sum_{n=2}^{N-1}   (C_h^{n-\frac{1}{2}}  w_h^{n}, \delta v_h^{n+1})_{\Omega}  + (C_h^{N-\frac{1}{2}}w_h^{N}, v_h^{N})_{\Omega}  \\&- (C_h^{\frac{1}{2}}w_h^1, v_h^{2})_{\Omega}+\sum_{n=1}^{N-1} ( (C_h^{n-\frac{1}{2}}-C_h^{n+\frac{1}{2}})  w_h^n, v_h^{n+1})_{\Omega}.
\end{alignat*}
Moreover, with $\delta C_h^{n+\frac{1}{2}}:= C_h^{n+\frac{1}{2}}-C_h^{n-\frac{1}{2}}$
\begin{equation*}
 \sum_{n=1}^{N-1} ( \delta  C_h^{n+\frac{1}{2}}  w_h^n, v_h^{n+1})_{\Omega} \le    Co C_{\gamma,\beta}  \sum_{n=1}^{N} \| w_h^n\| \|v_h^n\|
\end{equation*}
where $C_{\gamma,\beta} := C (1+\gamma_u) \frac{\beta_{1,\infty}}{1+\beta_{\infty}}$.
\end{lemma}
\begin{proof}
\begin{alignat*}{1}
\sum_{n=1}^{N-1}  (C_h^{n+\frac{1}{2}} \delta w_h^{n+1}, v_h^{n+1})_{\Omega}=&   \sum_{n=1}^{N-1}  \Big( (C_h^{n+\frac{1}{2}}  w_h^{n+1}, v_h^{n+1})_{\Omega}-  (C_h^{n+\frac{1}{2}} w_h^{n}, v_h^{n+1})_{\Omega} \Big)\\
=&  \sum_{n=2}^{N}   (C_h^{n-\frac{1}{2}}  w_h^{n}, v_h^{n})_{\Omega} - \sum_{n=1}^{N-1} (C_h^{n+\frac{1}{2}} w_h^{n}, v_h^{n+1})_{\Omega} \\
=& - \sum_{n=2}^{N-1}   (C_h^{n-\frac{1}{2}}  w_h^{n}, \delta v_h^{n+1})_{\Omega}+(C_h^{N-\frac{1}{2}}w_h^{N}, v_h^{N})_{\Omega}  \\
& - (C_h^{\frac{1}{2}}w_h^1, v_h^{2})_{\Omega}+ \sum_{n=1}^{N-1} ((C_h^{n-\frac{1}{2}}-C_h^{n+\frac{1}{2}}) w_h^{n}, v_h^{n+1})_{\Omega}.
\end{alignat*}
Moreover, we see that 
\begin{align*}
\sum_{n=1}^{N-1} ((C_h^{n-\frac{1}{2}}-&C_h^{n+\frac{1}{2}}) w_h^{n}, v_h^{n+1})_{\Omega}  =   \sum_{n=1}^{N-1} ( (\beta^{n-\frac{1}{2}}-\beta^{n+\frac{1}{2}}) \cdot \nabla w_h^n, v_h^{n+1})_{\Omega} \\&+{ \gamma_u \sum_{n=1}^{N-1} \sum_{F \in \mathcal{F}^{int}_h} \int_F h_F^2 (|\beta^{n-\frac{1}{2}}\cdot n|-|\beta^{n+\frac{1}{2}} \cdot n|) 
\jump{\nabla w^n_h} : \jump{\nabla v^{n+1}_h} ~\mbox{d}s}.  
\end{align*}
 This proves the identity. The proof of the inequality follows from the fact that $|(|\beta^{n-\frac{1}{2}} \cdot n|-|\beta^{n+\frac{1}{2}} \cdot n|)| \leq |(\beta^{n-\frac{1}{2}}-\beta^{n+\frac{1}{2}}) \cdot n|$ and $\beta^{n-\frac{1}{2}}-\beta^{n+\frac{1}{2}}= \int_{t_{n-\frac{1}{2}}}^{t_{n+\frac{1}{2}}} \partial_t \beta(t) dt$ and the inverse estimate \eqref{inverse}.
\end{proof}

Finally, we will need the following result. 
\begin{lemma}\label{Chlemma2}
For any $v_h^n \in V_h$ it holds
\begin{equation*}
\sum_{n=2}^{N-1} (C_h^{n-\frac{1}{2}} \delta v_h^{n}, 
\delta  \bar v_h^{n+1})_{\Omega}= -\frac12 \sum_{n=2}^{N-1} (C_h^{n-\frac{1}{2}} \delta  \delta v_h^{n+1}, \delta \bar v_h^{n+1} 
)_{\Omega} + \sum_{n=2}^{N-1}|\delta \bar v_h^{n+1}|_{s_u}^2.
\end{equation*}
\end{lemma}
\begin{proof}
 Using (\ref{Chsemi}) and the relation 
\[
\delta  \bar v_h^{n+1} -\delta v_h^{n} = \frac12 \delta
v_h^{n+1}+\frac12 \delta v_h^{n}-\delta v_h^{n} = \frac12 \delta
v_h^{n+1}-\frac12 \delta v_h^{n} = \frac12\delta \delta v_h^{n+1},
\]
we obtain that
\begin{align*}
     \sum_{n=2}^{N-1} (C_h^{n-\frac{1}{2}} \delta v_h^{n}, 
\delta  \bar v_h^{n+1})_{\Omega} &=  \sum_{n=2}^{N-1} (C_h^{n-\frac{1}{2}} (\delta v_h^{n}-\delta  \bar v_h^{n+1}), \delta  \bar v_h^{n+1} )_{\Omega} + \sum_{n=2}^{N-1}|\delta  \bar v_h^{n+1}|_{s_u}^2\\ & = -\frac12 \sum_{n=2}^{N-1} (C_h^{n-\frac{1}{2}} \delta  \delta v_h^{n+1}, \delta \bar v_h^{n+1} 
)_{\Omega} + \sum_{n=2}^{N-1}|\delta \bar v_h^{n+1}|_{s_u}^2.
\end{align*}
The result now follows by combining the above estimates. 
\end{proof}

The following approximation results are standard. For those related to the stabilization terms we refer to \cite{BFH006}. For simplicity, we assume
that the functions to approximate are smooth enough. We also use $\pi_h$ to denote the $L^2$-orthogonal
projection onto $W_h$ as well as that onto $V_h$. Consider the following approximate estimate applicable to quasi-uniform meshes:
\begin{lemma}(Approximation). \label{Approximation} Let $k \geq 1$ be the polynomial degree. Assume that $v \in [H^{k+1}]^d$ and $q \in H^{k+1}$. Then, the following holds when $Re>1$:
\begin{align*}
    \norm{v - \pi_h v} + h \norm{\nabla (v - \pi_h v)}&\leq C h^{k+1} |v|_{[H^{k+1}]^d},\\
    |v - \pi_h v|_{s_u} & \leq C \beta_{\infty}h^{k+\frac{1}{2}} \norm{v}_{[H^{k+1}(\Omega)]^d},\\
    \norm{q- \pi_h q} &\leq C h^{k+1} |q|_{H^{k+1}},\quad
    |q- \pi_h q|_{s_p} \leq  C h^{k+\frac12} \norm{ q}_{H^{k+1}(\Omega)}.
    \end{align*}
\end{lemma}
For the time discretization part of the error analysis we need some well known results on truncation
error analysis of finite difference operators. 
\begin{lemma} It holds, 
\begin{subequations} \label{estimates}
    \begin{alignat}{1} 
    \norm{\delta u^{n+1} -\tau \partial_t u(t^{n+\frac12})}^2 &\leq C \tau^5 \norm{\partial_{ttt}u}_{L^{2}((t^n,t^{n+1});L^{2}(\Omega))}^2, \label{estimates1}\\
    \norm{\bar u^{n+1} -u(t^{n+\frac12})}^2 &\leq C \tau^3 \norm{\partial_{tt}u}_{L^{2}((t^n,t^{n+1});L^{2}(\Omega))}^2,\label{estimates2} \\ 
    \norm{\hat u^{n+1} -u(t^{n+\frac12})}^2 &\leq C \tau^3 \norm{\partial_{tt}u}_{L^{2}((t^n,t^{n+1});L^{2}(\Omega))}^2. \label{estimates3}
 \end{alignat}
 \end{subequations}
\end{lemma}
\begin{proof}
Standard results for Crank-Nicolson approximation. Proof using Taylor development and Cauchy-Schwarz inequality. For details in slightly different norms we refer to \cite[Theorem 1.6]{Thom97}.
\end{proof}

Finally, we will need the following form of the discrete Gronwall's inequality that is also standard material and whose proof is omitted.
\begin{proposition}\label{Gronwall}
Let $\phi_n$ be a sequence of non-negative numbers and let $\psi$ and $\eta$ be non-negative numbers such that 
$
    \phi_N \leq \psi +\eta\sum^{N-1}_{i=1}\phi_i.
$
Then, the following estimate holds
$
    \phi_N \leq (1+N\eta \e^{\eta N})\psi.
$
\end{proposition}

\section{Stability and error analysis, Crank-Nicolson method with extrapolation}\label{sec:stab_error}
We below give the stability results for the Crank-Nicolson scheme (\ref{eq:momentum_CN})--(\ref{eq:mass_CN}). 
 
Here, we define the triple norm by: 
\[
|||v,p|||^2 := \tau \sum_{n=1}^{N-1} ( \|
 v^{n+1}\|_{E}^2 +\seminorm{p^{n+1}}^{2}_{s_p}).
\]
The following semi-norm on the second order increments of the solution will also be useful for the Crank-Nicolson analysis:

\[
|(u_h,p_h)|_{\delta \delta}^2 := \sum_{n=2}^{N-1} \|\delta  \delta
u_h^{n+1}\|^2   +\tau |\delta p_h^{N}|^2_{s_p}+ \tau \sum_{n=3}^{N-1}|\delta  \delta p_h^{n+1}|^2_{s_p}. 
\]
We will make use of the following relation,
\begin{equation}\label{eq:hat_bound}
\hat u_h^{n+1} =\bar u_h^{n+1}-\frac{1}{2}\delta \delta u_h^{n+1}.
 \end{equation}
\subsection{Stability Crank-Nicolson, general case}

Before we prove the stability result we need to estimate the pressure in terms of
the velocity.

\begin{lemma} \label{Ghlemma_CN}It holds that 
\begin{align} 
\sum^{N-1}_{n=1}\|G_h p_h^{n+1}\|^2 
\lesssim  \sum^{N-1}_{n=1}  \left({\frac{\mu}{h^2}} \| \bar u_h^{n+1}\|_A^2 + \frac{1}{h^2_{\beta}}  \| \hat u_h^{n+1}\|^2 +   {\frac{1}{\tau}}  \norm{\tilde{L}^{n}}^2_{\mathrm{div}}  + {\frac{1}{\tau}} \|L^{n+1}\|_h^2 \right).
\end{align}
\end{lemma}
\begin{proof}
We test the equation \eqref{eq:momentum_CN} with $G_h p_h^{n+1}$ to get
\begin{alignat*}{1}
  \|G_h p_h^{n+1}\|^2=& -\tau^{-1}(\delta u_h^{n+1}, G_h p_h^{n+1})_{\Omega} -(C_h^{n+\frac{1}{2}} \hat u_h^{n+1}, G_h p_h^{n+1})_{\Omega} \\
 &-(A_h \bar u_h^{n+1}, G_h p_h^{n+1})_{\Omega}+ L^{n+1}(G_h p_h^{n+1}).
\end{alignat*}
Using \eqref{estimate_Ch} we get
\begin{equation*}
-\sum^{N-1}_{n=1}(C_h^{n+\frac{1}{2}} \hat u_h^{n+1}, G_h p_h^{n+1})_{\Omega} \le   \sum^{N-1}_{n=1} \frac{C}{ 4 \epsilon h_{\beta}^2} \| \hat u_h^{n+1}\|^2 + \epsilon \sum^{N-1}_{n=1}\|G_h p_h^{n+1}\|^2.
\end{equation*}
Using \eqref{invA} we have
\begin{alignat*}{1}
    -\sum^{N-1}_{n=1} (A_h \bar u_h^{n+1}, G_h p_h^{n+1})_{\Omega}  &\le   C \frac{\mu}{4 \epsilon h^2} \sum^{N-1}_{n=1} \| \bar u_h^{n+1}\|_A^2 +  \epsilon \sum^{N-1}_{n=1}\|G_h p_h^{n+1}\|^2.
\end{alignat*}
Using \eqref{Gh} and Lemma \ref{lem:sum_by_parts_spec}, with $G_h^*$ instead of $C_h$, we obtain
\begin{align}\label{sp_delta}
-\sum^{N-1}_{n=1}(\delta u_h^{n+1}, G_h p_h^{n+1})_{\Omega}&= \sum^{N-1}_{n=1}( G_h^* \delta u_h^{n+1},  p_h^{n+1})_{\Omega} \nonumber
\\&= -\sum^{N-1}_{n=2}( G_h^*  u_h^{n}, \delta  p_h^{n+1})_{\Omega}+( G_h^*  u_h^{N},  p_h^{N})_{\Omega}-( G_h^*  u_h^{1},  p_h^{2})_{\Omega}.
\end{align}
By the relation (\ref{eq:mass_CN}) the first term of (\ref{sp_delta}) may be written as
\begin{align*}
  -\sum^{N-1}_{n=2}( G_h^*  u_h^{n}, \delta  p_h^{n+1})_{\Omega} =-\sum^{N-1}_{n=2} \tilde{L}^{n}( \delta  p_h^{n+1})+\sum^{N-1}_{n=2}s_p(   p_h^{n}, \delta  p_h^{n+1}). 
\end{align*}
Using the continuity of $\tilde L$ followed by Cauchy-Schwarz and Young's inequalities,
\begin{align*}
 -\sum^{N-1}_{n=2} \tilde{L}^n( \delta  p_h^{n+1}) &\le \sum_{n=2}^{N-1}\norm{\tilde{L}^{n}}_{\mathrm{div}}|\delta p_h^{n+1}|_{s_p}\\ &\le  \sum_{n=2}^{N-1}\norm{\tilde{L}^{n}}^2_{\mathrm{div}}+\frac14 \sum_{n=2}^{N-1}|\delta p_h^{n+1}|^2_{s_p}.
\end{align*}
{
For the $s_p(p_h^{n},\delta  p_h^{n+1})$ term recall that $2 b(a-b) = a^2 - (a-b)^2 - b^2$ to obtain
\begin{align*}
\sum^{N-1}_{n=2}s_p(   p_h^{n}, \delta  p_h^{n+1}) &= \frac12 \sum^{N-1}_{n=2}(s_p(   p_h^{n+1},   p_h^{n+1})-s_p(  \delta p_h^{n+1},  \delta p_h^{n+1})-s_p(   p_h^{n},   p_h^{n})) \\&= \frac12 s_p(p_h^N,p_h^N)-\frac{1}{2}\sum^{N-1}_{n=2}s_p(  \delta p_h^{n+1},  \delta p_h^{n+1})-\frac12 s_p(p_h^2,p_h^2).
\end{align*}
}
The last two terms of (\ref{sp_delta}) are handled by using (\ref{eq:mass_CN}),
\begin{align*}
  ( G_h^*  u_h^{N},  p_h^{N})_{\Omega} -( G_h^*  u_h^{1},  p_h^{2})_{\Omega}  &=\tilde{L}^{N}(p_h^N)-s_p(p_h^N,p_h^N)-\tilde{L}^1(p_h^2). 
\end{align*}
Using the continuity of $\tilde {L}$  we see that
\begin{align*}
 \tilde{L}^{N}(p_h^N)-\tilde{L}^1(p_h^2) \le \frac12 |p_h^N|^2_{s_p}+2  \norm{\tilde{L}^{N}}^2_{\mathrm{div}}+\frac12 |p_h^2|^2_{s_p}+2 \norm{\tilde{L}^{1}}^2_{\mathrm{div}}.
\end{align*}
Finally, using \eqref{enery_est} we obtain
\begin{alignat*}{1}
 \sum^{N-1}_{n=1} L^{n+1}(G_h p_h^{n+1}) \le &  \sum^{N-1}_{n=1} \|L^{n+1}\|_h (\|G_h p_h^{n+1}\|_E^2+ \|G_h p_h^{n+1}\|^2)^{\frac{1}{2}} \\
\le & \frac{1}{\tau}(4 \epsilon (C_E+\tau))\sum^{N-1}_{n=1} \|L^{n+1}\|_h^2 + \epsilon \sum^{N-1}_{n=1}\|G_h p_h^{n+1}\|^2.  
\end{alignat*}
Then by summing the above inequalities and noting that the terms $|p_h^2|_{s_p}$ and $|p_h^N|_{s_p}$ cancel we conclude.
\end{proof}

The analysis also requires stronger control of the increments in $|(u_h,v_h)|_{\delta\delta}$.

\begin{lemma}\label{lem:time_diss} Let $u_h$ solve (\ref{eq:momentum_CN}). If  $Re >1$ and  ${Co_{max} := \max} \{Co,Co_{4/3}\}\leq1$. Let $k \geq 1$, then the following estimate holds
\begin{align} \label{eq_1212}
|(u_h,p_h)|_{\delta\delta}^2  &
 \lesssim {\tau \sum_{n=2}^{N}\norm{ L^{n}}_h^2}+ \tau\sum_{n=1}^{N}\norm{\tilde{L}^{n}}^{2}_{\mathrm{div}} \nonumber\\& + \tau Co_{max} (1 + C_{\gamma,\beta}) \sum_{n=1}^{N-1}\|u_h^{n}\|^2 +{Co\Tnorm{\bar u_h, p_h}^2}+\tau | p_h^{2}|^2_{s_p}.
\end{align}
Moreover, if $Re>1$ and $Co$ is sufficiently small, then the following estimate holds
\begin{align} \label{eq_1212_1}
|(u_h,p_h)|_{\delta\delta}^2 &
  \lesssim {\tau \sum_{n=2}^{N-1}\norm{ L^{n+1}}_h^2}+ \tau\sum_{n=0}^{N-1}\norm{\tilde{L}^{n+1}}^{2}_{\mathrm{div}}+ \tau Co C_{\gamma, \beta}\sum_{n=2}^{N-1}\|u_h^{n}\|^2 \nonumber \\ &+{Co^{\frac12}\Tnorm{\bar u_h, p_h}^2}+Co^{\frac12} \sum^{N-1}_{n=2}\norm{ \delta  u_h^{n+1}-\pi_0\delta  u_h^{n+1}}^2.
\end{align}
\end{lemma}
\begin{proof} 
Multiplying (\ref{eq:momentum_CN}) by $\tau$, we see that {for $n \ge 2$}
\begin{align*}
\delta  \delta
u_h^{n+1} &= \delta u_h^{n+1} - \delta
u_h^{n} \\ &= -\tau C^{n-\frac{1}{2}}_h \delta \hat u_h^{n+1} -\tau (\delta C^{n+\frac{1}{2}}_h) \hat u_h^{n+1} -\tau A_h \delta \bar u_h^{n+1} - \tau G_h \delta p_h^{n+1} +{\tau \delta L^{n+1}}.
\end{align*}
Multiplying with $\delta  \delta
u_h^{n+1}$  and integrating, 
\begin{align}\label{eq:increm}
&\|\delta  \delta
u_h^{n+1}\|^2  \nonumber \\ &=
-\tau (C^{n-\frac{1}{2}}_h \delta \hat u_h^{n+1},\delta  \delta
u_h^{n+1})_{\Omega} - \tau (\delta C^{n+\frac{1}{2}}_h\hat u_h^{n+1},\delta  \delta
u_h^{n+1})_{\Omega} -\tau (A_h \delta  \bar u_h^{n+1},\delta  \delta
u_h^{n+1})_{\Omega} \nonumber \\ &- \tau (G_h \delta p_h^{n+1},\delta  \delta
u_h^{n+1})_{\Omega}+\tau \delta L^{n+1}(\delta  \delta
u_h^{n+1}).
\end{align}

Therefore, we easily have
$
\sum_{n=2}^{N-1}  \|\delta  \delta
u_h^{n+1}\|^2 =\sum_{i=1}^{5}S_i,
$
where
\begin{align*}
    &S_1 =-\tau \sum_{n=2}^{N-1} (C^{n-\frac{1}{2}}_h \delta \hat u_h^{n+1},\delta  \delta
u_h^{n+1})_{\Omega},  &S_2 = -\tau \sum_{n=2}^{N-1} (\delta C^{n+\frac{1}{2}}_h\hat u_h^{n+1},\delta  \delta
u_h^{n+1})_{\Omega}, \\
&S_3 =-\tau \sum_{n=2}^{N-1} (A_h \delta  \bar u_h^{n+1},\delta  \delta
u_h^{n+1})_{\Omega},  &S_4 =- \tau \sum_{n=2}^{N-1}(G_h \delta p_h^{n+1},\delta  \delta
u_h^{n+1})_{\Omega}, \\
&S_5 =\tau \sum_{n=2}^{N-1} \delta L^{n+1}(\delta  \delta
u_h^{n+1}).
\end{align*}
We start with an estimate of $S_2$. Using the Cauchy--Schwarz inequality, inverse inequality, and Young's inequality, 
\begin{align*}
   \tau (\delta C^{n+\frac{1}{2}}_h&\hat u_h^{n+1},\delta  \delta u_h^{n+1})_{\Omega} \\ &\leq \tau( |(\beta(t^{n+\frac{1}{2}})-\beta(t^{n-\frac{1}{2}})) \nabla \hat u_h^{n+1}, \delta  \delta
u_h^{n+1})_{\Omega}| \\ &+{ \tau \gamma_u  \Bigl|\sum_{F \in \mathcal{F}^{int}_h} \int_F h_F^2 |(\beta^{n+\frac{1}{2}}-\beta^{n-\frac{1}{2}}) \cdot n|
\jump{\nabla \hat u^{n+1}_h} : \jump{\nabla \delta  \delta u^{n+1}_h} ~\mbox{d}s}\Bigr|\\ &\leq  CCo^2 (4 \epsilon)^{-1} \tau \beta^2_{1, \infty}/(1+\beta_{\infty})^2 \norm{\hat u^{n+1}_h}^2 +\epsilon \norm{\delta  \delta u_h^{n+1}}^2.
\end{align*}
Taking the summation over $n$= 2, 3, ..., $N$-1,
\begin{align*}
 S_2 & \leq  Co C_{\gamma,\beta} \tau  \sum_{n=2}^{N-1}\norm{\hat u^{n+1}_h}^2 +\epsilon \sum_{n=2}^{N-1}\norm{\delta  \delta u_h^{n+1}}^2 \\
 & \leq  Co C_{\gamma,\beta} \tau  \sum_{n=2}^{N-1}\norm{u^{n}_h}^2 +\epsilon \sum_{n=2}^{N-1}\norm{\delta  \delta u_h^{n+1}}^2.
\end{align*}
The term $S_3$ is handled by using the Cauchy--Schwarz inequality, (\ref{invA}) and  $Re>1$,
\begin{align*}
  S_3 \le \tau\sum_{n=2}^{N-1} \norm{\delta  \bar u_h^{n+1}}_A \norm{\delta  \delta
u_h^{n+1}}_A & \le C\tau\sum_{n=2}^{N-1} \norm{ \delta  \bar u_h^{n+1}}_A \frac{\sqrt{\mu}}{h}\norm{\delta  \delta
u_h^{n+1}} \\ &\le  C\sqrt{\tau}\sum_{n=2}^{N-1} \norm{ \delta  \bar u_h^{n+1}}_A  \left(\frac{Co}{Re}\right)^{\frac{1}{2}}\norm{\delta  \delta
u_h^{n+1}} \\ & \leq   C Co \epsilon^{-1} {\Tnorm{\bar u_h, 0}^2} + \epsilon \sum_{n=2}^{N-1}\norm{\delta  \delta u_h^{n+1}}^2.
\end{align*}
For $S_4$,
\begin{align*}
S_4 = - \tau \sum_{n=2}^{N-1}(G_h \delta p_h^{n+1},\delta  \delta
u_h^{n+1})_{\Omega} = \tau \sum_{n=2}^{N-1}( \delta p_h^{n+1},G_h^* \delta  \delta
u_h^{n+1})_{\Omega} \\
= \tau \sum_{{n=3}}^{N-1}\delta \delta \tilde L^{n+1}(\delta p_h^{n+1})-\tau \sum_{n=3}^{N-1}s_p( \delta p_h^{n+1}, \delta  \delta
p_h^{n+1})_{\Omega} + \tau (\delta p_h^{3},G_h^*  \delta  \delta
u_h^{3})_{\Omega}.
\end{align*}
For the first term on the right hand side we note that using Lemma \ref{lem:sum_by_parts_spec} there holds
\[
 \sum_{n=3}^{N-1}\delta \delta \tilde L^{n+1}(\delta p_h^{n+1}) = - \sum_{n=4}^{N-1} \delta \tilde L^n(\delta \delta p_h^{n+1}) + \delta \tilde L^N(\delta p^N) - \delta \tilde L^3(\delta p^4),
\]
and hence
$
\tau \sum_{n=3}^{N-1}\delta \delta \tilde L^{n+1}(\delta p_h^{n+1}) \leq
\epsilon |(0,p_h)|_{\delta\delta} + \tau\sum_{n=2}^{N}\norm{\tilde{L}^{n}}^{2}_{\mathrm{div}}.
$

Now consider the term
\begin{align*}
(\delta p_h^{3},G_h^*  \delta  \delta
u_h^{3})_{\Omega} &= (\delta p_h^{3},G_h^*  (\delta u_h^3- \delta u_h^2))_{\Omega}\\&=(\delta p_h^{3},G_h^*  \delta u_h^3)_{\Omega}-(\delta p_h^{3},G_h^*    u_h^2)_{\Omega}+(\delta p_h^{3},G_h^*    u_h^1)_{\Omega} \\ &= {\tilde{L}^3(\delta p^3_h)}-s_p(\delta p_h^{3},\delta p_h^3) -{2 \tilde{L}^2(\delta p^3_h)}+  s_p(\delta p_h^{3},p_h^2) + \underbrace{(\delta p_h^{3},G_h^* u_h^1)}_{\tilde L^1(\delta p_h^3)}.
\end{align*}
\[{\tilde{L}^3(\delta p^3_h)-2 \tilde{L}^2(\delta p^3_h)} \le \epsilon^{-1}(\norm{\tilde{L}^2}^2_{\mathrm{div}}+\norm{\tilde{L}^3}^2_{\mathrm{div}})+\epsilon|\delta p^3_h|^2_{s_p}.\]
\[
s_p(\delta p_h^{3},p_h^2) = s_p(p_h^{3}-p_h^2,p_h^2) = \frac 12 s_p(p_h^{3},p_h^{3}) - \frac12 s_p(\delta p_h^{3},\delta p_h^3) -\frac 12 s_p(p_h^{2},p_h^{2}).
\]
Now note that since $(a - b) a = \frac12 (a^2 + (a-b)^2 - b^2)$ we have using telescoping property that 
 \begin{align*}
 \sum_{n=3}^{N-1}s_p(\delta  \delta p_h^{n+1},\delta p_h^{n+1} ) &= \frac12 s_p(\delta p_h^{N},\delta p_h^{N} ) + \frac12 \sum_{n=3}^{N-1}s_p(\delta  \delta p_h^{n+1},\delta \delta p_h^{n+1} ) - \frac12 s_p(\delta p_h^{3},\delta p_h^{3} ) \\ &={ \frac12 |\delta p_h^{N}|^2_{s_p}+ \frac12 \sum_{{n=3}}^{N-1}|\delta  \delta p_h^{n+1}|^2_{s_p} - \frac12 |\delta p_h^{3}|^2_{s_p}}.
 \end{align*}
Collecting terms we see that
\begin{align*}
-\tau \sum_{n=3}^{N-1}s_p( \delta p_h^{n+1}, \delta  \delta
p_h^{n+1})_{\Omega}  & + \tau (\delta p_h^{3},G_h^*  \delta  \delta
u_h^{3})_{\Omega} 
 \leq -\tau \frac12 |\delta p_h^{N}|^2_{s_p} -\tau \frac12 \sum_{n=3}^{N-1}|\delta  \delta p_h^{n+1}|^2_{s_p}   \\ & -\tau|\delta p_h^{3}|^2_{s_p}+ \tau \frac12 |p_h^3|_{s_p}^2- \tau \frac12 |p_h^2|_{s_p}^2   +\tau L^1_{div}(\delta p_h^3).
\end{align*}
{
To control the $s_p$ terms note that for all $\varepsilon>0$ there holds
\begin{align*}
 -|\delta p_h^{3}|^2_{s_p} + \frac12 |p_h^3|_{s_p}^2- \frac12 |p_h^2|_{s_p}^2  
 & = -|p_h^3|_{s_p}^2 -|p_h^2|_{s_p}^2 + 2 s_p(p_h^3,p_h^2)+ \frac12 |p_h^3|_{s_p}^2- \frac12 |p_h^2|_{s_p}^2 \\
 & = -\frac12 |p_h^3|_{s_p}^2 -\frac32 |p_h^2|_{s_p}^2 + 2 s_p(p_h^3,p_h^2) \\
 & \leq -\Bigl(\frac12 - \frac{1}{\varepsilon}\Bigr)  |p_h^3|_{s_p}^2 -\Bigl(\frac32 - \varepsilon \Bigr) |p_h^2|_{s_p}^2.
\end{align*}
Here we used that
$
2 s_p(p_h^3,p_h^2) \leq \frac{1}{\epsilon} |p_h^3|_{s_p}^2 + \epsilon |p_h^2|_{s_p}^2
$
and we see that with $\varepsilon =3$, $$-\Bigl(\frac12 - \frac{1}{\varepsilon}\Bigr)  |p_h^3|_{s_p}^2 -\Bigl(\frac32 - \varepsilon \Bigr) |p_h^2|_{s_p}^2\leq  -\frac{1}{6} |p_h^3|_{s_p}^2 + \frac32 |p_h^2|_{s_p}^2. $$
}
It follows that
\begin{align*}
\tau \sum_{n=3}^{N-1}\delta \delta \tilde L^{n+1}(\delta p_h^{n+1}) & -\tau \sum_{n=3}^{N-1}s_p( \delta p_h^{n+1}, \delta  \delta
p_h^{n+1})_{\Omega} + \tau (\delta p_h^{3},G_h^*  \delta  \delta
u_h^{3})_{\Omega} 
\\ &\leq -\left(\frac12 - \epsilon\right) |(0,p_h)|_{\delta\delta}^2   + \frac32 |p_h^2|_{s_p}^2   + C \tau\sum_{n=1}^{N}\norm{\tilde{L}^{n}}^{2}_{\mathrm{div}}.
\end{align*}


The term $S_5$ is handled by using (\ref{enery_est}),
\begin{align*}
  S_5 &\le  \sum_{n=2}^{N-1} \tau  \norm{L^{n+1}}_h(\norm{\delta  \delta u_h^{n+1}}_E^2+\norm{\delta  \delta u_h^{n+1}}^2)^\frac{1}{2}\\ & \leq  \sum_{n=2}^{N-1} \tau^{\frac{1}{2}}\norm{L^{n+1}}_h(C_E+\tau+1)^{\frac{1}{2}}\norm{\delta  \delta u_h^{n+1}} \\&\leq  C \tau \epsilon^{-1} \sum_{n=2}^{N-1}\norm{L^{n+1}}^2_h  + \epsilon \sum_{n=2}^{N-1}\norm{\delta  \delta u_h^{n+1}}^2.
\end{align*}
It remains to bound the $S_1$ term. To do this
observe that
\begin{align*}
   S_1 &=-\tau \sum_{n=2}^{N-1} (C^{n-\frac{1}{2}}_h \delta \hat u_h^{n+1},\delta  \delta
u_h^{n+1})_{\Omega} \\ &=-\frac{3}{2}\tau \sum_{n=2}^{N-1} (C^{n-\frac{1}{2}}_h \delta  u_h^{n},\delta  \delta
u_h^{n+1})_{\Omega}+\frac{1}{2}\tau \sum_{n=2}^{N-1} (C^{n-\frac{1}{2}}_h \delta u_h^{n-1},\delta  \delta
u_h^{n+1})_{\Omega}.
\end{align*}
Consider the first term on the right hand side.
\begin{align*}
  -\tau (C^{n-\frac{1}{2}}_h \delta  u_h^{n},\delta  \delta
u_h^{n+1})_{\Omega} =T^{n}_1+T^{n}_2+T^{n}_3+T^{n}_4,
\end{align*}
where
\begin{align*}
    T^{n}_1:&=-\tau^2 ( C^{n-\frac{1}{2}}_h L_h^n,\delta  \delta
u_h^{n+1}),\  \quad 
&T^{n}_3:=\tau^2 (C^{n-\frac{1}{2}}_hA_h^n \bar u^n_h,\delta  \delta
u_h^{n+1}),\\
T^{n}_2:&=\tau^2 (C^{n-\frac{1}{2}}_hC_h^{n-\frac{1}{2}} \hat u^n_h,\delta  \delta
u_h^{n+1}),\  \quad
&T^{n}_4:=\tau^2 (C^{n-\frac{1}{2}}_hG_h^n p^n_h,\delta  \delta
u_h^{n+1}).
\end{align*}
Using inverse estimates (\ref{estimate_Ch}) and (\ref{enery_est}) we obtain
\begin{align*}
   T^{n}_1 & \le C {\tau}^2 h^{-1}_{\beta}  \norm{L^n}_h(\norm{\delta  \delta
u_h^{n+1}}^2_E+\norm{\delta  \delta
u_h^{n+1}}^2)^{\frac{1}{2}} \\  &\le C \sqrt{\tau} (4 \epsilon)^{-1} Co \norm{L^n}_h(C_E+\tau)^{\frac{1}{2}}\norm{\delta  \delta u_h^{n+1}}\leq \tau Co^2 \norm{L^n}^2_h +\epsilon  \norm{\delta  \delta u_h^{n+1}}^2.
\end{align*}
To estimate $T^{n}_2$ we apply \eqref{estimate_Ch} twice to obtain
\begin{align}\nonumber
 T^{n}_2  \leq C \frac{\tau^2}{h_\beta^2} \norm{\hat u_h^n} \norm{\delta  \delta u_h^{n+1}} 
 & \leq C (4 \epsilon)^{-1} \left( \frac{\tau}{h_\beta}\right)^4\norm{\hat u_h}^2 + \epsilon\norm{\delta  \delta u_h^{n+1}}^2 \\
  & \leq C (4 \epsilon)^{-1} \tau Co_{4/3}^3 \norm{\hat u_h}^2 + \epsilon\norm{\delta  \delta u_h^{n+1}}^2.\label{s1_est}
 \end{align}

Using (\ref{invA}) and (\ref{estimate_Ch}) we obtain
\begin{align*}
  T^{n}_3 \le \tau^2 \norm{\bar u^n_h}_A \norm{C^{(n-\frac{1}{2})*}_h \delta  \delta
u_h^{n+1}}_A &\le C\tau^2 \sqrt{\mu} \beta_\infty^{-1} h^{-2}_{\beta} \norm{\bar u^n_h}_A  \norm{\delta  \delta
u_h^{n+1}}  \\
& \leq C (4 \epsilon)^{-1} Co^3 Re^{-1} \beta^{-1}_\infty \tau \norm{\bar u^n_h}_A^2 + \epsilon \norm{\delta  \delta
u_h^{n+1}}^2.
\end{align*}
Finally, again using the inverse estimate (\ref{estimate_Ch}) we obtain
\begin{align}\label{eq_gp}
    T^{n}_4\le \tau^2 \norm{C^{n-\frac{1}{2}}_hG_h^n p^n_h} \norm{\delta  \delta
u_h^{n+1}} &\le \tau^2 h^{-1}_{\beta}\norm{G_h^n p^n_h} \norm{\delta  \delta
u_h^{n+1}}\nonumber \\&\le \epsilon^{-1}\tau^4 h^{-2}_{\beta}\norm{G_h^n p^n_h}^2+ \epsilon\norm{\delta  \delta
u_h^{n+1}}^2.
\end{align}
Using Lemma \ref{Ghlemma_CN} and the fact that $Re>1$,  we obtain (with the constant depending on $\epsilon$),
\begin{align}
\sum^{N-1}_{n=2}T^{n}_4 &\leq  C \tau \Bigl(  Co^3 \sum^{N-1}_{n=2} \| \bar u_h^{n}\|_A^2 + Co_{4/3}^3  \sum^{N-1}_{n=2}\| \hat u_h^{n}\|^2  +Co^2\sum^{N-1}_{n=1}  \norm{\tilde{L}^{n}}^2_{\mathrm{div}}  \nonumber\\&  + Co^2\sum^{N-1}_{n=2} \|L^{n}\|_h^2  \Bigr)+\epsilon \sum^{N-1}_{n=2}\norm{\delta  \delta
u_h^{n+1}}^2. 
\end{align}
Hence,  with $Co_{\max}=\max\{Co, Co_{4/3}\} \leq 1$, and use that $\norm{\hat u^{n}_h} \lesssim 3/2 \norm{u^{n-1}_h} + 1/2 \norm{u^{n-2}_h}$
\begin{align}\label{eq_240622}
    \frac{3}{2}\tau \sum_{n=2}^{N-1} (C^{n-\frac{1}{2}}_h \delta  u_h^{n}&,\delta  \delta
u_h^{n+1})_{\Omega}   \leq 4 \epsilon \sum_{n=2}^{N-1}\norm{\delta  \delta u_h^{n+1}}^2
     + Co_{\max} \tau \sum_{n=1}^{N-2} \norm{ u^{n}_h}^2 \\&
        + C Co \Bigl( \Tnorm{\bar u_h,  p_h}^2\nonumber+ \tau\sum_{n=2}^{N-1}\norm{L^{n}}^{2}_h+\tau\sum^{N-1}_{n=1}  \norm{\tilde{L}^{n}}^2_{\mathrm{div}} \Bigr).
\end{align}
Similarly, we can derive the same estimate for $\frac{1}{2}\tau \sum_{n=2}^{N-1} (C^{n-\frac{1}{2}}_h \delta  u_h^{n-1},\delta  \delta
u_h^{n+1})_{\Omega}$ and so we get
\begin{align*}
     S_1 &\leq 8 \epsilon \sum_{n=2}^{N-1}\norm{\delta  \delta u_h^{n+1}}^2
     +2  Co_{\max} \tau \sum_{n=1}^{N-2} \norm{u^{n}_h}^2 \\&
        + 2 C Co \Bigl( \Tnorm{\bar u_h,  p_h}^2\nonumber+ \tau\sum_{n=2}^{N-1}\norm{L^{n}}^{2}_h+\tau\sum^{N-1}_{n=3}  \norm{\tilde{L}^{n}}^2_{\mathrm{div}}+\tau \norm{\tilde{L}^{1}}^2_{\mathrm{div}} \Bigr).
\end{align*}
If we combine the estimates of terms $S_i$ and take $\epsilon$ small enough we obtain (\ref{eq_1212}).

Let us now prove the estimate (\ref{eq_1212_1}). Note that we wish to eliminate the $Co_{4/3}$ dependence that appeared through the estimation of $S_1$. Therefore we once again consider $S_1$ but this time we add and subtract the piecewise constant function $\pi_0 \delta \hat u_h^{n}$,
\begin{align*}
    -&\tau (C^{n-\frac{1}{2}}_h \delta \hat u_h^{n+1},\delta  \delta
u_h^{n+1})_{\Omega} \\ &=\underbrace{-\tau (C^{n-\frac{1}{2}}_h (\delta \hat u_h^{n+1}-\pi_0 \delta \hat u_h^{n+1}),\delta  \delta u_h^{n+1})_{\Omega}}_{I}-\underbrace{\tau(\beta_{\infty} \pi_0 \delta \hat u_h^{n+1} \cdot n,\delta  \delta u_h^{n+1} \cdot n)_{\partial \Omega}}_{II}.
\end{align*}
The first term is handled by using the Cauchy--Schwarz inequality, inverse estimates \ref{estimate_Ch} and Young's inequality, 
\begin{align*}
    I \lesssim Co   \norm{\delta \hat u_h^{n+1}-\pi_0 \delta \hat u_h^{n+1}}^2+ Co \norm{\delta  \delta u_h^{n+1}}^2.
\end{align*}
The second term can be written as:
\begin{align}\label{57ii}
    II=& \tau(\beta_{\infty} (\pi_0 \delta \hat u_h^{n+1}-\delta \hat u_h^{n+1}) \cdot n,\delta  \delta u_h^{n+1} \cdot n)_{\partial \Omega}\\
    &+\tau(\beta_{\infty} (\delta \hat u_h^{n+1} - \delta \bar u_h^{n+1}) \cdot n,\delta  \delta u_h^{n+1} \cdot n)_{\partial \Omega}  + \tau(\beta_{\infty} (\delta \bar u_h^{n+1} \cdot n,\delta  \delta u_h^{n+1} \cdot n)_{\partial \Omega}. \nonumber
\end{align}
Using the Cauchy--Schwarz inequality, trace inequality \ref{tracePk} and the Young's inequality in the first term on the right hand side of \ref{57ii} we obtain
\begin{align*}
-\tau(\beta_{\infty} (\pi_0 \delta \hat u_h^{n+1}-\delta \hat u_h^{n+1}) \cdot n,&\delta  \delta u_h^{n+1} \cdot n)_{\partial \Omega}  \\&\le \tau \norm{\beta_{\infty} (\pi_0 \delta \hat u_h^{n+1}-\delta \hat u_h^{n+1}) \cdot n}_{\partial \Omega} \norm{\delta  \delta u_h^{n+1} \cdot n}_{\partial \Omega} \\ &\le \tau \beta_{\infty} h^{-1}\norm{ \pi_0 \delta \hat u_h^{n+1}-\delta \hat u_h^{n+1} } \norm{\delta  \delta u_h^{n+1} } \\  &\lesssim Co \norm{ \pi_0 \delta \hat u_h^{n+1}-\delta \hat u_h^{n+1} }^2 + Co \norm{\delta  \delta u_h^{n+1} }^2.
\end{align*}
The second term is handled using the relation \eqref{eq:hat_bound},
\begin{align*}
    \tau (\beta_{\infty} (\delta \hat u_h^{n+1} - \delta \bar u_h^{n+1} )\cdot n,& \delta  \delta u_h^{n+1} \cdot n)_{\partial \Omega}\\ &\le \tau \beta_{\infty} ( \norm{ \delta \delta u_h^{n+1} \cdot n}_{\partial \Omega} + \norm{ \delta \delta u_h^{n} \cdot n}_{\partial \Omega} ) \norm{\delta  \delta u_h^{n+1} \cdot n}_{\partial \Omega} \\ &\le C\beta_\infty \tau h^{-1}(\norm{\delta  \delta u_h^{n+1} }^2+ \norm{\delta  \delta u_h^{n} }) \\ &\le C Co(\norm{\delta  \delta u_h^{n+1} }^2+\norm{\delta  \delta u_h^{n} }^2).
\end{align*}
Finally the third term can be bounded as follows 
\begin{align*}
\tau \beta_{\infty} (\delta \bar u_h^{n+1} \cdot n,\delta \delta u_h^{n+1} \cdot n)_{\partial \Omega} & \leq Co^{\frac12} \tau^{\frac12} (|\bar u_h^{n+1}|_{s_u} + |\bar u_h^{n}|_{s_u}) \|\delta  \delta u_h^{n+1}\| \\
& \lesssim Co^{\frac12}  \tau (|\bar u_h^{n+1}|_{s_u}^2+ |\bar u_h^{n}|_{s_u}^2) + Co^{\frac12}  \|\delta  \delta u_h^{n+1}\|^2.
\end{align*}
Thus,
\[
    S_1 \lesssim  Co\sum^{N-1}_{n=2} \norm{\delta  u_h^{n+1}-\pi_0\delta  u_h^{n+1}}^2+  Co^{\frac12} \sum^N_{n=2}\norm{\delta \delta u_h^n}^2 + Co^{\frac12} \sum^N_{n=2} |\bar u_h^n|_{s_u}^2. 
\]
If we combine this bound with the previous estimates of $S_i$, $i=2,...,5$ and choose $Co$ sufficiently small then we obtain (\ref{eq_1212_1}).
\end{proof}
 
\begin{lemma}\label{lem:pres_s_inc}
There holds for all $\epsilon > 0$,
\[
\tau \sum_{n=2}^{N-1}  s_p(\delta p_h^{n+1},\delta p_h^{n+1}) \leq  \frac{\epsilon}{2} \tau \sum_{n=2}^{N-1} |p_h^{n+1}|_{s_p}^2 + \epsilon^{-1}  |(0, p_h^{n+1})|_{\delta\delta}^2 + (1+\epsilon^{-1}) \tau |p_h^{2}|_{s_p}^2.
\]
\end{lemma} 
 \begin{proof}
Observe that using Lemma \ref{lem:sum_by_parts_spec} we have
\[
\sum_{n=2}^{N-1}  s_p(\delta p_h^{n+1},\delta p_h^{n+1}) = -\sum_{n=3}^{N-1}  s_p( p_h^{n+1},\delta \delta p_h^{n+1}) + s_p( p_h^{N},\delta p_h^{N}) - s_p( p_h^{2},\delta p_h^{3}).
\]
It follows that using Cauchy-Schwarz inequality followed by Young's inequality, for any $\epsilon>0$,
\begin{align*}
\sum_{n=2}^{N-1}  s_p(\delta p_h^{n+1},\delta p_h^{n+1}) \leq 
\frac{\epsilon}{4} \sum_{n=3}^{N-1}  s_p( p_h^{n+1}, p_h^{n+1})
+ \epsilon^{-1} \sum_{n=3}^{N-1}  s_p( \delta \delta p_h^{n+1}, \delta \delta p_h^{n+1}) \\
+\frac{\epsilon}{4} s_p(   p_h^{N},  p_h^{N})+ \epsilon^{-1} s_p( \delta  p_h^{N}, \delta p_h^{N}) - s_p( p_h^{2},\delta p_h^{3}).
\end{align*}
Consider the last term of the right hand side
\[
- s_p( p_h^{2},\delta p_h^{3}) =  s_p( p_h^{2},p_h^2 - p_h^{3}) \leq \frac{\epsilon}{4} s_p( p_h^{3},p_h^3) + (1+\epsilon^{-1}) s_p( p_h^{2},p_h^2). 
\]
 \end{proof}
 
 We can now establish the stability results of the extrapolated Crank-Nicolson scheme. Compared to the BDF2 instance, the Crank-Nicolson scheme requires careful modification, as it lacks the same favourable dissipation characteristics as the BDF2 approach.
  \begin{theorem}
Suppose that $k\ge 1$, $Re >1$ and $Co_{\max}={\max} \{Co,Co_{4/3}\}$ is sufficiently small only
depending on geometric constants of the mesh and the parameters $\gamma$ and $\gamma_u$. Let $T=N \tau$. For $\{u^n_h\}$
 solving (\ref{eq:momentum_CN})--(\ref{eq:mass_CN}) we have the
following bound:
\begin{align}\label{th_CN_1}
  \max_{2 \leq n \leq N} \norm{u^{n}_h}^{2}   \leq C(\norm{u^{0}_h}^2+\norm{u^{1}_h}^2&
  +{\tau \sum_{n=1}^{N-1} \norm{ L^{n+1}}_h^2}+ \tau\sum_{n=1}^{N}\norm{\tilde{L}^{n}}^{2}_{\mathrm{div}})M,
\end{align}
  where $M$ is given by 
\begin{align} \label{th2_m}
    M= C\left(1+T (1+ \phi) \e^{T(1+\phi)} \right), \quad \ \text{and} \ \phi =(1+\frac{ \beta_{1,\infty}}{1+\beta_\infty}).
\end{align}
In addition, it holds
\begin{align}\label{th1p2}
 { \Tnorm{\bar u_h,p_h}^2 }  \leq  \phi T \big(\|u_h^0\|^2 &+ \| u_h^1\|^2+
\tau \sum_{n=1}^{N-1}\norm{L^{n+1}}_h^2+ \tau\sum_{n=1}^{N}\norm{\tilde{L}^{n}}^{2}_{\mathrm{div}}\big)M.
\end{align}
\end{theorem} 
\begin{proof}
Multiplying \eqref{eq:momentum_CN} with $\bar u_h^{n+1}=
(u_h^{n+1}+u_h^{n})/2$, integrating and summing over $n$ and using (\ref{Chsemi}) yield
\begin{multline}\label{eq:stab_1_CN1}
 \frac{1}{2}(\norm{u^{N}_h}^{2}-\norm{u^{1}_h}^{2}) +  \tau\sum_{n=1}^{N-1}  (G_h p_h^{n+1}, \bar u_h^{n+1} )_{\Omega} +
\tau \sum_{n=1}^{N-1}(C_h^{n+{\frac{1}{2}}} (\hat u_h^{n+1} - \bar u_h^{n+1}),\bar
u_h^{n+1})_{\Omega} \\ +\tau\sum_{n=1}^{N-1} \|\bar u_h^{n+1} \|_E^2 = {\tau \sum_{n=1}^{N-1} L^{n+1}(\bar u_h^{n+1})}.
\end{multline}
{For the pressure we use the equations (\ref{Gh}), (\eqref{eq:mass_CN}) and the
  definition of the adjoint for $n\ge2$
\begin{align}\label{pres_aveg}
(G_h p_h^{n+1}, \bar u_h^{n+1} )_{\Omega} = -  (p_h^{n+1}, G_h^* \bar
 u_h^{n+1} )_{\Omega} = {(S^p_h \bar p_h^{n+1},  p_h^{n+1})_{\Omega}}-\tilde L^{n+1}(p_h^{n+1}). 
\end{align}
For $n=1$, using \eqref{eq:mass_CN} and the definition of $\tilde{L}$,
\begin{align*}
  (G_h p_h^{2}, \bar u_h^{2} )_{\Omega} =-( p_h^{2}, G_h^*\bar u_h^{2} )_{\Omega}&=-\frac12( p_h^{2},  G_h^*u_h^{1} )_{\Omega}-\frac12( p_h^{2}, G_h^* u_h^{2} )_{\Omega} \\&= -\frac12 \tilde{L}^1(p^2_h)-\frac12 \tilde{L}^2(p^2_h)+\frac12 |p^2_h|^2_{s_p}.
\end{align*}}
Then the equation (\ref{eq:stab_1_CN1}) becomes
\begin{multline}\label{eq:stab_1_CN}
 \frac{1}{2}(\norm{u^{N}_h}^{2}-\norm{u^{1}_h}^{2}) + \tau \sum_{n=2}^{N-1}  {(S_h \bar p_h^{n+1},  p_h^{n+1})_{\Omega}}  + \tau\sum_{n=1}^{N-1}\|\bar u_h^{n+1} \|_E^2 =\sum_{i=1}^{3}R^{n+1}_i.
\end{multline}
\begin{align*}
    R^{n+1}_1:&=\tau \sum_{n=1}^{N-1} L^{n+1}(\bar u_h^{n+1}),\\
    R^{n+1}_2:&=-\tau\sum_{n=1}^{N-1} (C_h^{n+{\frac{1}{2}}} (\hat u_h^{n+1} - \bar u_h^{n+1}),\bar
u_h^{n+1})_{\Omega}, \\
R^{n+1}_3 &:={\tau\frac12 (\tilde L^1(p_h^2)+\tilde L^2(p_h^2))}+\tau\sum^{N-1}_{n=2}{\tilde{L}^{n+1}(p_h^{n+1})}.
\end{align*}
Let us estimate $R^{n+1}_1$ and $R^{n+1}_3$. We have
\begin{align*}
  R^{n+1}_1 &\leq   \tau \sum_{n=1}^{N-1} \norm{L^{n+1}}_h\sqrt{\norm{\bar u_h^{n+1}}_E^2+\norm{\bar u_h^{n+1}}^2}\\
      & \leq 8\tau\sum^{N-1}_{n=1}\norm{L^{n+1}}^2_h + \frac{1}{32}\tau\sum^{N-1}_{n=1}\norm{\bar u_h^{n+1}}^2_E+\frac{1}{32}\tau\sum^{N-1}_{n=1}\norm{ u_h^{n+1}}^2.
\end{align*}
and
{ \begin{alignat}{1}
    R^{n+1}_3 &\leq \tau\frac12 (\norm{\tilde L^1}_{\mathrm{div}}|p_h^2|_{s_p}+\norm{\tilde L^2}_{\mathrm{div}}|p_h^2|_{s_p})+\tau\sum^{N-1}_{n=2}\norm{\tilde{L}^{n+1}}_{\mathrm{div}} |p^{n+1}_h|_{s_p} \nonumber \\ 
    & \leq \frac{1}{\epsilon} \tau\norm{\tilde{L}^{1}}^2_{\mathrm{div}}+{\frac{1}{2\epsilon} \tau\sum^{N-1}_{n=1}\norm{\tilde{L}^{n+1}}^2_{\mathrm{div}} +\frac{\epsilon}{2}\tau\sum^{N-1}_{n=1}|p_h^{n+1}|_{s_p}^2}. \label{R3}
\end{alignat}}
For the $R^{n+1}_2$ term we use the relation \eqref{eq:hat_bound}
\[
\hat u_h^{n+1} - \bar u_h^{n+1} = -\frac12 \delta \delta u_h^{n+1},
\]
together with Lemma \ref{lem:sum_by_parts_spec} to obtain
\begin{multline}\label{eq:c_bound_1_CN}
-\sum_{n=1}^{N-1} (C_h^{n+{\frac{1}{2}}} (\hat u_h^{n+1} - \bar u_h^{n+1}), \bar
u_h^{n+1})_{\Omega} = \frac12 \sum_{n=1}^{N-1} (C_h^{n+{\frac{1}{2}}} \delta \delta u_h^{n+1}, \bar
u_h^{n+1})_{\Omega} \\
\leq -\frac12  \sum_{n=2}^{N-1} (C_h^{n-\frac{1}{2}} \delta u_h^{n},
\delta  \bar u_h^{n+1})_{\Omega} + Co C_{\gamma, \beta} \sum_{n=1}^{N-1} \| \delta u_h^{n}\| \|\bar
u_h^{n}\|\\ + \frac12  (C_h(t^{N-\frac{1}{2}})
\delta u_h^{N}, \bar
u_h^{N})_{\Omega} - \frac12 (C_h(t^1) \delta u_h^{1}, \bar u_h^{2})_{\Omega}.
\end{multline}
Using Lemma \ref{Chlemma2} in the first term on the right hand of (\ref{eq:c_bound_1_CN}) we have
\begin{align}\label{eq_2406}
\frac12 \sum_{n=2}^{N-1} (C_h^{n-\frac{1}{2}} \delta u_h^{n}, 
\delta  \bar u_h^{n+1})_{\Omega} = -\frac14 \sum_{n=2}^{N-1} (C_h^{n-\frac{1}{2}} \delta  \delta u_h^{n+1}, \delta \bar u_h^{n+1} 
)_{\Omega} + \frac12 \sum_{n=2}^{N-1}|\delta \bar u_h^{n+1}|_{s_u}^2.
\end{align}
Multiplying the equation (\ref{eq:c_bound_1_CN}) with $\tau$ and using  (\ref{eq_2406}) leads to
\begin{multline} \label{eq_2406_1}
R^{n+1}_2 
\leq \frac14\underbrace{\tau \sum_{n=2}^{N-1} (C_h^{n-\frac{1}{2}} \delta  \delta u_h^{n+1}, \delta \bar u_h^{n+1} 
)_{\Omega} - \tau\frac12 \sum_{n=2}^{N-1}|\delta \bar u_h^{n+1}|_{s_u}^2 }_{\Xi}\\ +\tau \frac12  (C_h(t^{N-\frac{1}{2}})
\delta u_h^{N}, \bar
u_h^{N})_{\Omega} +\tau Co C_{\gamma, \beta}  \sum_{n=1}^{N-1}\| \delta u_h^{n}\| \|\bar
u_h^{n}\| - \tau\frac12 (C_h(t^1) \delta u_h^{1}, \bar u_h^{2})_{\Omega}.
\end{multline}
Using the similar arguments as in (\ref{eq_240622}) in the first term on the right hand side we obtain
\begin{align*}
    \Xi  &\leq CCo \sum_{n=2}^{N-1}\norm{\delta  \delta
u_h^{n+1}}^2 +Co_{\max} \tau \sum_{n=2}^{N-1} \norm{\hat u^{n+1}_h}^2+ {Co \Tnorm{\bar u_h,  p_h}^2}\nonumber\\&+ CCo\tau\sum_{n=2}^{N-1}\norm{L^{n+1}}^{2}_h+ CoC\tau\sum_{n=1}^{N-1}\norm{\tilde{L}^{n+1}}^{2}_{\mathrm{div}}.
\end{align*}
The last two terms are bounded as:
\begin{align*}
 \tau (C_h&(t^{N-\frac{1}{2}}) \delta u_h^{N}, \bar u_h^{N})_{\Omega} -  \tau (C_h(t^1) \delta u_h^{1},  \bar u_h^{2})_{\Omega} \\ &\leq C  Co (\norm{\delta u^N_h}^2+\norm{  \bar { u}^{N}_h}^2+\norm{\delta u^1_h}^2+\norm{  \bar { u}^{2}_h}^2).
\end{align*}
 Taking $Co \le 1$ and $Co_{\max}=\max\{Co, Co_{4/3}\}$, $C_{\gamma, \beta}= \frac{ (1+\gamma_u) C_i \beta_{1,\infty}}{1+\beta_\infty}$, the equation (\ref{eq_2406_1}) becomes,
\begin{align} \label{T2}
    R^{n+1}_2  &\leq CCo \sum_{n=2}^{N-1}\norm{\delta  \delta
u_h^{n+1}}^2+ {Co \Tnorm{\bar u_h,  p_h}^2 }+CCo\tau \sum_{n=2}^{N-1}\norm{L^{n+1}}^{2}_h\nonumber\\& +\tau Co_{\max} C_{\gamma, \beta}\sum_{n=1}^{N-1}\|  u_h^{n}\|^2+ CoC\tau\sum_{n=1}^{N-1}\norm{\tilde{L}^{n+1}}^{2}_{\mathrm{div}} \nonumber\\&+ C  Co (\norm{\delta u^N_h}^2+\norm{  \bar { u}^{N}_h}^2+\norm{\delta u^1_h}^2+\norm{  \bar { u}^{2}_h}^2).
\end{align}
The equation (\ref{eq:stab_1_CN}) becomes
\begin{multline*}
 \frac{1}{2}(\norm{u^{N}_h}^{2}-\norm{u^{1}_h}^{2}) + \tau \sum_{n=2}^{N-1}  {(S_h  \bar p_h^{n+1},  p_h^{n+1})_{\Omega}}  + \tau\sum_{n=1}^{N-1}\|\bar u_h^{n+1} \|_E^2 \\ \lesssim Co \sum_{n=2}^{N-1}\norm{\delta  \delta
u_h^{n+1}}^2+ {Co \Tnorm{\bar u_h,  p_h}^2}+\tau \sum_{n=1}^{N-1}\norm{L^{n+1}}^{2}_h \nonumber\\ +\tau Co_{\max} C_{\gamma, \beta}\sum_{n=1}^{N-1} \|  u_h^{n}\|^2 + \tau\sum_{n=1}^{N}\norm{\tilde{L}^{n+1}}^{2}_{\mathrm{div}}\\ +  Co (\norm{ u^N_h}^2+\norm{{ u}^{N-1}_h}^2+\norm{u^1_h
}^2+\norm{ { u}^{2}_h}^2).
\end{multline*}
For the pressure term in the right hand side we observe that since $(a - b) a = \frac12 (a^2 + (a-b)^2 - b^2)$ we have using telescoping property that
{\begin{align*}
\sum_{n=2}^{N-1} s_p( \bar p_h^{n+1},  p_h^{n+1}) &= \sum_{n=2}^{N-1} s_p(  p_h^{n+1},  p_h^{n+1}) -\frac12 \sum_{n=2}^{N-1} s_p( \delta p_h^{n+1},   p_h^{n+1}) \\&=\sum_{n=2}^{N-1}  |p_h^{n+1}|^2_{s_p}+\frac14   |p_h^{2}|^2_{s_p}-\frac14 \sum_{n=2}^{N-1}  |\delta p_h^{n+1}|^2_{s_p} - \frac14 |p_h^N|^2_{s_p}.
\end{align*}}
As a consequence we have
\begin{multline*}
 \norm{u^{N}_h}^{2} + |||(\bar u_h, p_h)|||^2  \\ \lesssim  \norm{u^{1}_h}^{2}+Co \sum_{n=2}^{N-1}\norm{\delta  \delta
u_h^{n+1}}^2+ Co \Tnorm{\bar u_h,  p_h}^2 +\tau \sum_{n=1}^{N-1}\norm{L^{n+1}}^{2}_h \nonumber\\ +\tau Co_{\max} C_{\gamma, \beta} \sum_{n=1}^{N-1} \|  u_h^{n}\|^2+ \tau\sum_{n=1}^{N}\norm{\tilde{L}^{n}}^{2}_{\mathrm{div}}\\+   Co (\norm{ u^N_h}^2+\norm{{ u}^{N-1}_h}^2+\norm{ { u}^{2}_h}^2) +{\tau \sum_{n=2}^{N-1}  |\delta p_h^{n+1}|^2_{s_p}}.
\end{multline*}

Taking $Co$ small and applying Lemma \ref{lem:pres_s_inc} to the last term of the right hand side we may then write
\begin{multline} 
\norm{u^{N}_h}^{2} +\Tnorm{\bar u_h,  p_h}^2  \lesssim \norm{u^{1}_h}^{2}+|(u_h,p_h)|_{\delta \delta}^2 +\tau \sum_{n=1}^{N-1}\norm{L^{n+1}}^{2}_h \\ +\tau Co_{\max} C_{\gamma, \beta} \sum_{n=1}^{N-1} \|  u_h^{n}\|^2+ \tau\sum_{n=0}^{N-1}\norm{\tilde{L}^{n+1}}^{2}_{\mathrm{div}}\\+   Co (\norm{ u^N_h}^2+\norm{{ u}^{N-1}_h}^2+\norm{ { u}^{2}_h}^2)+ \tau | p_h^{2}|^2_{s_p}. 
\end{multline}
Applying Lemma \ref{lem:time_diss} we can bound the term $|(u_h,p_h)|_{\delta\delta}$, for $Co$ small enough leading to
\begin{multline}
\norm{u^{N}_h}^{2} +\Tnorm{\bar u_h,  p_h}^2  \lesssim \tau \sum_{n=1}^{N-1}\norm{L^{n+1}}^{2}_h \\ +\tau Co_{\max} C_{\gamma,\beta} \sum_{n=1}^{N-1} \|  u_h^{n}\|^2+ \tau\sum_{n=0}^{N-1}\norm{\tilde{L}^{n+1}}^{2}_{\mathrm{div}}\\+   Co (\norm{ u^N_h}^2+\norm{{ u}^{N-1}_h}^2+\norm{ { u}^{2}_h}^2)+  \frac12\tau | p_h^{2}|^2_{s_p} .
\end{multline}
Observe that the terms with $u_h^2,\, u_h^{N-1}$ and $p_h^2$  in the right hand side remain to be bounded.
First to eliminate the $u_h^{N},\, u_h^{N-1},\, u_h^{2}$ contributions let $u_h^{n*}$ be such that $\|u_h^{n*}\| = \max_{n \in \{2,\hdots,N\}} \|u_h^n\|$.
Since $\norm{ u^N_h}^2+\norm{{ u}^{N-1}_h}^2+\norm{ { u}^{2}_h}^2 \leq 3 \|u_h^{n*}\|^2$ it follows that
 for $Co$ small enough
\[
\norm{u^{n*}_h}^{2} -   C  Co (\norm{ u^N_h}^2+\norm{{ u}^{N-1}_h}^2+\norm{ { u}^{2}_h}^2) \ge  (1 - 3 C Co) \norm{u^{n*}_h}^{2} \ge \frac12 \norm{u^{n*}_h}^{2}.
\]
Applying the above bound with $N = n*$ we see that the desired inequality holds for $\norm{u^{n*}_h}^{2}$. Now extend the summation in the right hand side to $N-1$. Since $\norm{u^n_h}^2 \leq \norm{u^{n*}_h}^2$ for all $2 \leq n \leq N$, the relation holds for all $n$ and we obtain

\begin{multline*}
\max_{2 \leq n \leq N} \norm{u^{n}_h}^{2}  \lesssim  \norm{u^{1}_h}^2 +\tau \sum_{n=2}^{N}\norm{L^{n}}^{2}_h + \tau\sum_{n=1}^{N}\norm{\tilde{L}^{n}}^{2}_{\mathrm{div}} +\tau \phi\sum_{n=2}^{N-1} \|  u_h^{n}\|^2
+  \frac12\tau | p_h^{2}|^2_{s_p} .
\end{multline*}
For $N=2$ and $Co$, $\tau$ sufficiently small we obtain the bound for $| p_h^{2}|^2_{s_p}$, in terms of $\norm{u^{1}_h}^2$ and $\norm{u^{0}_h}^2$ by testing \eqref{eq:momentum_CN} by $v_h = \bar u_h^2$ and \eqref{eq:mass_CN} by $q_h = p_h^2$.
The estimate for \eqref{th_CN_1} follows from Gronwall's inequality: Proposition \ref{Gronwall}. We can then use the above inequality to get the estimate \eqref{th1p2}. 
\end{proof}
\subsection{Improved stability for piece wise affine elements and CIP stabilization, Crank-Nicolson}
The following result proves the stability of the Crank--Nicolson scheme under the standard hyperbolic CFL for
piecewise affine approximation.

\begin{lemma}\label{lemmaY_1}
Let $k=1$,  $y_h^{n}= \delta u_h^{n}$  and  $Y_h^{n}:= y_h^n-\pi_0 y_h^n$ then
\begin{alignat}{1} \label{lemmaY_k1}
 \|Y_h^{n+1}\|^2 &\le   C \Bigl( \frac{Co}{Re} \tau  \|\bar u_h^{n+1}\|_A^2 +  \tau^2  \beta_{1,\infty}^2 \|\hat u_h^{n+1}\|^2+ \tau  Co |\bar u_h^{n+1}|_{s_u}^2 \nonumber \\
  & +    \tau Co   |p_h^{n+1}|_{s_p}^2 +   \tau \|L_h^{n+1}\|_h^2+Co \norm{\delta \delta u^{n+1}_h}^2 \Bigr). 
\end{alignat}
\end{lemma}
\begin{proof}
This proof is similar to \cite[Lemma 5.3]{BGG23}, but we sketch it here for completeness.
Using \eqref{eq:momentum_CN} we obtain
$\|Y_h^{n+1}\|^2=(y_h^{n+1}, \pi_h Y_h^{n+1})_{\Omega} = \tau (M_1+M_2+M_3+M_4),$
where
\begin{alignat*}{1}
M_1:=&-(C_h^{n+\frac12} \hat u_h^{n+1},  \pi_h Y_h^{n+1})_{\Omega}, \quad  M_2:=-(A_h \bar u_h^{n+1},  \pi_h Y_h^{n+1} )_{\Omega}, \\
M_3:=& -( G_h p_h^{n+1},  \pi_h Y_h^{n+1})_{\Omega}, \quad  M_4:= L^{n+1}(\pi_h Y_h^{n+1}).
\end{alignat*}
We first estimate the $M_2$ and $M_4$ which are the most simple ones to bound. 
Using \eqref{invA} we obtain 
$ 
M_2 \le  C  (\frac{Co}{\tau Re})^{\frac{1}{2}}  \|\bar u_h^{n+1}\|_A \|Y_h^{n+1}\|.
$ 
We also easily see using \eqref{enery_est} that
\begin{equation*}
M_4 \le  \frac{(C_E +\tau)^{\frac{1}{2}}}{\tau^{\frac{1}{2}}} \|L^{n+1}\|_h   \|Y_h^{n+1}\|.
\end{equation*}
To estimate $M_1$ we use the definition of $C_h^{n+\frac12}$ to write 
\begin{equation*}
M_1:=- ( \beta^{n+\frac12} \cdot \nabla \hat u_h^{n+1}, \pi_h Y_h)_{\Omega}+ s_u( \hat u_h^{n+1},   \pi_h Y_h^{n+1}). 
 \end{equation*}
We have using \eqref{stab_bound}
$ 
s_u( \hat u_h^{n+1},   \pi_h Y_h) \le C (\frac{Co}{\tau})^{\frac{1}{2}} |\hat u_h^{n+1}|_{s_u} \| Y_h^{n+1}\|.   
$ 
On the other hand, for an arbitrary $w_h \in V_h$ 
\begin{alignat*}{1}
   &( \beta^{n+\frac12} \cdot \nabla \hat u_h^{n+1}, \pi_h Y_h^{n+1})_{\Omega} \\
=& ( (\beta^{n+\frac12}-\beta_0^{n+\frac12}) \cdot \nabla \hat u_h^{n+1}, Y_h^{n+1})_{\Omega}   + ( \beta^{n+\frac12} \cdot \nabla \hat u_h^{n+1}-w_h,  (\pi_h-I) Y_h^{n+1})_{\Omega}.
\end{alignat*}
Note here that we crucially used that $\beta_0^{n+\frac12} \cdot \nabla \hat u_h^{n+1} \in X_h$.  
Thus, using \eqref{interpinq} and \eqref{inverse} we obtain
$ 
   ( \beta^{n+\frac12} \cdot \nabla \hat u_h^{n+1}, \pi_h Y_h^{n+1})_{\Omega}  \le  C \big(\beta_{1,\infty} \|\hat u_h^{n+1}\|+ \frac{1}{h_{\beta}^{\frac{1}{2}}} |\hat u_h^{n+1}|_{s_u}  \big) \|Y_h^{n+1}\|.
$ 
{Finally, by the triangle inequality and \eqref{eq:hat_bound} we have the bound
\begin{align}
|\hat u_h^{n+1}|_{s_u} \le |\bar u_h^{n+1}|_{s_u}+\frac{C\sqrt{Co}}{\sqrt{\tau}}\|\delta \delta u_h^{n+1}\|.
\end{align}}
According to the definition of $G_h$ and using that $\nabla p_h^{n+1} \in X_h$ we obtain
\begin{alignat*}{1}
M_3=& -( \nabla p_h^{n+1},  \pi_h Y_h^{n+1})_{\Omega}= ( \nabla p_h^{n+1},  (I-\pi_h) Y_h^{n+1})_{\Omega}\\
=& ( \nabla p_h^{n+1}-v_h,  (I-\pi_h) Y_h^{n+1})_{\Omega}, \quad   \forall v_h \in V_h.
\end{alignat*}
Thus, using  \eqref{pres_inter} we get
$ 
M_3 \le C h_\beta^{-\frac12} |p_h^{n+1}|_{s_p} \| Y_h^{n+1}\|.
$ 
Putting all the above estimates together we obtain (\ref{lemmaY_k1}). 

\end{proof}

\begin{theorem}
Suppose that $T=N \tau$, $Re >1$,  $\gamma_u>0$ $\gamma_p>0$ and $k=1$. Suppose that $Co$ is sufficiently small only
depending on geometric constants of the mesh and $\gamma_u$. For $\{u^n_h\}$
 solving (\ref{eq:momentum_CN})--(\ref{eq:mass_CN}) we have the
following bound:
\begin{align}\label{th_CN_2}
 \max_{2 \leq n \leq N} \norm{u^{n}_h}^{2}   \lesssim (\norm{u^{0}_h}^2+\norm{u^{1}_h}^2&+{\tau \sum_{n=1}^{N-1}\norm{L^{n+1}}_h^2}+ \tau\sum_{n=1}^{N}\norm{\tilde{L}^{n}}^{2}_{\mathrm{div}})M, 
  \end{align}
  where $M$ is given by 
  \begin{align} \label{th1_m}
    M=\left(1+cT\phi \e^{cT \phi}\right), \quad \phi=\left((1+\gamma_u)\frac{ \beta_{1,\infty}}{1+\beta_\infty} +  \tau \beta^2_{1,\infty} \right),
\end{align}
and
\begin{align}\label{th2p2}
  {\Tnorm{\bar u_h,p_h}^2 }  \lesssim \phi T \big(\|u_h^0\|^2 + \| u_h^1\|^2&+
\tau \sum_{n=1}^{N-1}\norm{L^{n+1}}_h^2+ \tau\sum_{n=1}^{N}\norm{\tilde{L}^{n}}^{2}_{\mathrm{div}}\big)M.
\end{align}
\end{theorem}
\begin{proof}  To get an improved estimate we use several of the estimates established in the previous proof. Indeed to avoid the introduction of the $Co_{4/3}$ factor, we only need to find an improved estimate for $\Xi$ in (\ref{eq_2406_1}). To this end, setting $y_h^{n+1} := \delta u_h^{n+1}$, $Y_h^{n+1}:= y_h^{n+1}- \pi_0 y_h^{n+1}$ and using the definition of $C_h$ \eqref{eq:Cdef} we see that
\begin{alignat}{1} \label{Xi_p1}
 \Xi= & \frac12\tau \sum_{n=2}^{N-1} (C_h^{n-\frac{1}{2}} \delta  \delta u_h^{n+1}, \delta \bar u_h^{n+1} 
)_{\Omega} - \tau \sum_{n=2}^{N-1}|\delta \bar u_h^{n+1}|_{s_u}^2. 
\end{alignat}
{
First note that
\begin{align*}
    \tau (C_h^{n-\frac{1}{2}} \delta  \delta u_h^{n+1}, \delta \bar u_h^{n+1} 
)_{\Omega}&=\tau (\beta^{n-\frac{1}{2}} \cdot \nabla \delta \delta u_h^{n+1},\delta \bar u_h^{n+1} )_{\Omega}+\tau s_u(\delta \delta u_h^{n+1},\delta \bar u_h^{n+1})
    \\ &=-\tau (\beta^{n-\frac12} \cdot \nabla \delta \bar u_h^{n+1}, \delta \delta u_h^{n+1} )_{\Omega}+\tau s_u(\delta \delta u_h^{n+1},\delta \bar u_h^{n+1})
    \\ &=-\tau \sum_{T \in \mathcal{T}} (\beta^{n-\frac{1}{2}} \cdot \nabla \bar Y_h^{n+1},\delta \delta u_h^{n+1})_{T}+\tau s_u(\delta \delta u_h^{n+1},\delta \bar u_h^{n+1}).
\end{align*}
The first term is handled by using that on each element $T \in \mathcal{T}$,  $\nabla y_h\vert_{T} = \nabla Y_h\vert_T$, the Cauchy--Schwarz inequality, inverse inequality \ref{inverse} and Young's inequality,
\begin{align*}
    \tau \sum_{T \in \mathcal{T}} (\beta^{n-\frac{1}{2}} \cdot \nabla \bar Y_h^{n+1},\delta \delta u_h^{n+1})_{T}
    \le Co \norm{  Y^{n+1}_h}^2+Co \norm{  Y^{n}_h}^2+Co \norm{\delta \delta u_h^{n+1}}^2.
\end{align*}
Now using, Lemma \ref{lemmaY_1} with $Re \ge 1$ we obtain 
\begin{alignat*}{1}
 \tau \sum_{n=2}^{N-1} \sum_{T \in \mathcal{T}} (\beta^{n-\frac{1}{2}} \cdot \nabla \bar Y_h^{n+1},\delta \delta u_h^{n+1})_{T} & \lesssim  Co^{\frac12} \sum_{n=2}^{N-1}  \| \delta \delta u_h^{n+1}\|^2 +Co^{\frac12} ||| \bar u_h, p_h|||^2  \\
 & + \tau^2  \beta_{1,\infty}^2 \sum_{n=1}^{N} \|u_h^{n}\|^2 +  \tau (C_E+ \tau) \sum_{n=2}^{N-1} \|L^{n+1}\|_h^2.  
\end{alignat*}
The second term can be estimated as:
\begin{align*}
    \tau  s_u(\delta \delta u_h^{n+1},\delta \bar u_h^{n+1})-\tau |\delta \bar u_h^{n+1}|_{s_u}^2  \leq \frac{\tau}{4}| \delta \delta u_h^{n+1}|_{s_u}^2 \lesssim Co \| \delta \delta u_h^{n+1}\|^2.
\end{align*}

 Combining these estimates  we obtain
\begin{alignat}{1}\label{eq_2006_1t}
 \Xi \lesssim   Co^{\frac{1}{2}} ||| \bar u_h, p_h|||^2+  \tau^2  \beta_{1,\infty}^2 \sum_{n=1}^{N} \|u_h^{n}\|^2 +  \tau (C_E+ \tau) \sum_{n=2}^{N-1} \|L^{n+1}\|_h^2.
 \end{alignat}
 }
Using (\ref{eq_2006_1t}) in (\ref{eq_2406_1}) we obtain
\begin{alignat*}{1}
    R_2  &\leq    C  Co (\norm{\delta u^N_h-\pi_0 \delta u^N_h}^2+\norm{  \bar { u}^{N}_h}^2+\norm{\delta u^1_h
-\pi_0 \delta u^1_h}^2+\norm{  \bar { u}^{2}_h}^2) \\&+  C \tau  \sum_{n=2}^{N-1} \|L_h^{n+1}\|_h^2 +C  Co {||| \bar u_h,  p_h|||^2 }\\&+    C (\tau  \beta_{1,\infty}^2+Co (1+\gamma_u) \frac{\beta_{1,\infty}}{1+\beta_{\infty}}) \tau \sum_{n=1}^{N} \|u_h^{n}\|^2+C Co^{\frac{1}{2}} \sum_{n=2}^{N-1} \| \delta \delta  u_h^{n+1}\|^2.
\end{alignat*}
Here we used that $C_E +\tau $ is bounded since  we are assuming that $Co \le 1$. 
{ Finally, using this inequality in (\ref{eq:stab_1_CN1}), controlling all terms except $\Xi$ as before, except that we now use equation \eqref{eq_1212_1} of Lemma \ref{lem:time_diss} followed by Lemma \ref{lemmaY_1} and the fact that $Co$ is sufficiently small 
\begin{align*}
   \max_{2 \leq n \leq N} \norm{u^{n}_h}^{2} 
   \lesssim&    \norm{u^0_h}^2+\norm{u^1_h}^2
   + \phi \tau \sum_{n=1}^{N-1} \|u_h^{n}\|^2 +  \tau  \sum_{n=2}^{N-1} \|L_h^{n}\|_h^2 \\&+ \tau\sum_{n=1}^{N-1}\norm{\tilde{L}^{n}}^{2}_{\mathrm{div}}+ Co^{\frac{1}{2}} \sum_{n=2}^{N-1} \| \delta \delta  u_h^{n+1}\|^2 +  \tau | p_h^{2}|^2_{s_p}.
\end{align*}
The last term is bounded by $\norm{u^0_h}^2+\norm{u^1_h}^2$ as before and we obtain
\begin{equation*}
\max_{2 \leq n \leq N} \norm{u^{n}_h}^{2} \lesssim  \norm{u^{0}_h}^2+\norm{u^{1}_h}^2 +\tau \sum_{n=2}^{N}\norm{L^{n}}^{2}_h + \tau\sum_{n=1}^{N}\norm{\tilde{L}^{n}}^{2}_{\mathrm{div}}.
\end{equation*}
}
We can now use the discrete Gronwall inequality to get \eqref{th_CN_2} and \eqref{th2p2} follows using the bound \eqref{th_CN_2} in the error equation.

\end{proof}
\subsection{A priori error estimate for the velocity approximation}
In this section we will study the error in the Crank--Nicolson IMEX method.  
Reacll that in \eqref{eq:momentum_CN} $p^{n+1}_h$ is an approximation of $p( t^{n+\frac{1}{2}})$ and $u_h^{n+1}$ is an approximation of $u(t^{n+1})$. We therefore define {$u^{n}=u(t^{n})$} and $p^{n+1}=p(t^{n+\frac12})$.
In order to prove error estimates, we first derive the error equations. We use the following notation:
\begin{alignat*}{2}
w^n_h=& \pi_h u^n,\quad   e^n_h=  w^n_h-u_h^n, \quad \eta^n_h=  w^n_h-u^n,\\
q^n_h=& \pi_h p^n, \quad \zeta^n_h=  q^n_h-p_h^n, \quad \kappa^n_h=  q^n_h-p^n.
\end{alignat*}

By the definition of $u_h^0$ and $u_h^1$ there holds $e_h^0 = e_h^1 = 0$. Then, we write the error equations as 
\begin{align}\label{eq_09}
(\delta e_h^{n+1},v_h)_\Omega &+ \tau (C_h^{n+\frac{1}{2}}\hat e_h^{n+1}, v_h)_{\Omega} \nonumber\\&+ \tau (A_h \bar e_h^{n+1}, v_h)_{\Omega} +\tau (G_h \zeta^{n+1}_h,v_h)_{\Omega}  
= \tau L^{n+1}(v_h),\quad \forall v_h \in V_h,\nonumber\\
&{(G^*_h  e_h^{n+1},y_h)_\Omega+(S^p_h \zeta^{n+1}_h,y_h)_\Omega}= \tilde{L}^{n+1}(y_h), \quad \forall y_h \in Q_h,
\end{align}
where
\begin{equation}\label{l_eq}
\tau  L^{n+1}(v_h):= \sum^{5}_{j=1} \Psi_j^{n+1}(v_h),\quad { \tilde{L}^{n+1}(y_h):= \sum^{2}_{j=1} \Phi_j^{n+1}(y_h)},
\end{equation}
and
\begin{align*}
    \Psi_1^{n+1}(v_h):&=(\delta u^{n+1} -\tau\partial_t u(t^{n+\frac12}),v_h)_\Omega, \quad \Psi_2^{n+1}(v_h):=\tau( G_h\kappa_h^{n+1}, v_h)_\Omega, \\ \Psi_3^{n+1}(v_h):&=\tau(C_{h}^{n+\frac{1}{2}} \hat \eta_h^{n+1},v_h)_\Omega, \quad \Psi_4^{n+1}(v_h)  :=\tau(C_{h} (\hat u^{n+1}- u(t^{n+\frac12})),v_h)_\Omega, \\
    \Psi_5^{n+1}(v_h):&= \tau(A_h \bar \eta_h^{n+1},v_h)_{\Omega}, \quad \Psi_6^{n+1}(v_h)  :=\tau(A_{h} (\bar u^{n+1}- u(t^{n+\frac12})),v_h)_\Omega,
    \\
    \Phi_1^{n+1}(y_h):&= (G^*_h  \eta_h^{n+1},y_h)_{\Omega}, \quad \Phi_2^{n+1}(y_h)  :=(S^p_h \kappa_h^{n+1},y_h)_{\Omega}.
\end{align*}

First we verify that $\tilde L$ is zero for all functions in the kernel of $S_h^p$. Assume $S_h^p q_h = 0$. Immediately we see that $\Phi_2^{n+1}(q_h)=0$. This means that $q_h \in C^1(\bar \Omega)$ and therefore $\nabla q_h \in V_h$. We conclude by observing that
\[
\Phi_1^{n+1}(q_h) = (\pi_h u^{n+1} - u^{n+1},\nabla q_h)_{\Omega} = 0.
\]
\begin{theorem} \label{error_cn}
Let $u$ be the solution of (\ref{eq111}) and $\{u^n_h\}^N_{n=0}$
 be the solution of (\ref{eq:momentum_CN})-(\ref{eq:mass_CN}). Let $T=N \tau$, $Re >1$. Suppose that $Co$  is sufficiently small when $k=1$ and $\max \{Co, Co_{4/3}\}$ is sufficiently small  when $k \geq 2$. Then for all $n \ge 2$,
 \begin{align} 
   \norm{\pi_h u(t^n)-u^n_h} &\leq C ( {\tau^{2}} +h^{k+\frac{1}{2}} )M, \label{eq_06}
\end{align}
where  the constant $C$ depends on the Sobolev norms in the right hand side of the approximation bounds of Lemma \ref{Approximation} and the truncation errors \eqref{estimates1}--\eqref{estimates3}. The latter are also applied to certain differential operators, with the highest order given by $\norm{  \partial_t^2 \Delta {u}}_{L^2((0,T);L^2(\Omega))}$. $M$ is given by \eqref{th1_m} in the case $k=1$ and \eqref{th2_m} in the case $k \ge 2$.
In addition, it holds
\begin{align}\label{th1p2_u}
  \Tnorm{\pi_h u-u_h,\pi_h p-p_h}   \leq  \sqrt{\psi T} C (\tau^2 +h^{k+\frac{1}{2}} )M.
\end{align}
\end{theorem}
\begin{proof} Using  \eqref{th_CN_1} or \eqref{th_CN_2} we have,
\begin{align}\label{mainthmaux}
    \|e_h^N\|^2 \leq  ({\tau \sum_{n=1}^{N-1}\norm{L^{n+1}}_h^2}+\tau \sum_{n=1}^{N} \|\tilde{L}^{n}\|_{\mathrm{div}}^2)M.
\end{align}
It remains to bound $\Psi_1$ -- $\Psi_6$ and $\Phi_1$ -- $\Phi_2$. The terms $\Psi_2$--$\Psi_5$ as well as $\Phi_1$--$\Phi_2$ are handled as in \cite[Theorem 5.5]{BGG23}. In particular $\Psi_2$, $\Psi_3$ and $\Phi_1$ use the stabilization bounds \eqref{interpinq}, \eqref{interdiv} and \eqref{pres_inter}. $\Psi_1$ and $\Psi_6$ are bounded using \eqref{estimates1} and \eqref{estimates2},
\begin{alignat*}{1}
\Psi_1^{n+1}(v_h) \lesssim  \tau^{\frac52} \|\partial^3_t{u}\|_{L^2(t_n, t_{n+1};[ L^2(\Omega)]^d)} \|v_h\|.
\end{alignat*}
\begin{align*}
    \Psi_6^{n+1}(v_h)  \lesssim \tau  \mu\| \Delta(\bar u^{n+1}- u(t^{n+\frac12})\| \| v_h\|  &\lesssim \mu \tau^{\frac52}  \norm{  \partial_t^2 \Delta {u}}_{L^2((t^n,t^{n+1});L^2(\Omega)}  \|v_h\|.
\end{align*}
Collecting the bounds of the $\Psi_i$, $i=1,...,6$ and $\Phi_i$, $i=1,...2$ in \eqref{mainthmaux} we arrive at \eqref{eq_06}. Using \eqref{th1p2} or \eqref{th2p2} we then deduce \eqref{th1p2_u}.

\end{proof}

\subsection{Error estimates for the pressure approximation}
In this section we will discuss error bounds for the pressure in the high Reynolds regime. First we state a sub-optimal bound in the $H^1$-norm. Then, for the time averaged pressure we give a bound of similar order as those for the velocities. The proofs of these results are straightforward combining the above error estimates with the arguments of the pressure error analysis of \cite[Proposition 5.6 and Theorem 5.7]{BGG23}.
\begin{proposition}
Assume that $h^2 \lesssim \tau$ then there holds for $N \ge 4$
\begin{align}
\|p^{N}-p^{N}_h\|^2_{H^1(\Omega)} \lesssim \sqrt{\psi T} ( \tau^2/h + h^{k-\frac{1}{2}} )M. 
\end{align}
\end{proposition}
The following result provides an error estimate for the average pressure.
\begin{theorem}
Let $p$ be the solution of (\ref{eq111}) and $\{p^n_h\}^N_{n=0}$
 be the solution of (\ref{eq:momentum_CN})--(\ref{eq:mass_CN}). Let $T=N \tau$, $Re >1$. Suppose that $Co$  is sufficiently small when $k=1$ and $\max \{Co, Co_{4/3}\}$ is sufficiently small  when $k \geq 2$. Then,
\begin{align}
   \norm{& \bar  p^N-\bar p_h^N} \leq C  h^{k+1}\norm{ \bar p^N}_{H^{k+1}(\Omega)}+   C Y\left( \tau^2 +h^{k+\frac{1}{2}} \right) M, \label{eq_07}
\end{align}
where
$
Y$ depends on $\beta_{\infty}$, $T$ and $\phi$.
The values of $C$, $\phi$, and $M$ the same as in Theorem \ref{error_cn}.
\end{theorem}

\section{Splitting methods derived from the IMEX method}\label{sec:split}
We showed in \cite[Section 6]{BGG23} that for the case of inviscid flow (or very high Reynolds
flow where the viscous dissipation may be treated explicitly) the IMEX method
 allows for a natural decoupling strategy. In algorithm \ref{alg:split} below we present such an algorithm for the IMEX Crank-Nicolson method. The proof is similar to that of \cite{BGG23} and therefore omitted.
\begin{algorithm}
 \flushleft
Given $u_h^{n}$ and $u_h^{n-1}$ we compute
\begin{enumerate}
\item Let $u^* = 1.5 u_h^{n} - 0.5 u_h^{n-1}$ and find $C_h u_h^* \in V_h$ such that
\[
(C_h u_h^*, v_h)_{\Omega} = (\beta \cdot \nabla u_h^*, v_h)_{\Omega} + s_u(u_h^*,v),
\]
for all $v_h \in V_h$.  
\item Let
\begin{equation}\label{eq:RC}
R_C(v) = -(C_h u_h^*, v)_{\Omega}+(\tau^{-1} u_h^{n},v)_{\Omega}.
\end{equation}
for all $v \in L^2(\Omega)$.
\item Solve for the pressure: Find $p_h^{n+1} \in Q_h$ such that
\[
(\nabla p_h^{n+1},\nabla q_h)_{\Omega} = R_C(\nabla q_h), \, \forall q_h \in Q_h.
\]
\item Solve for $u_h^{n+1} \in V_h$ such that
\[
(\tau^{-1} u_h^{n+1},v_h)_{\Omega}
= (p_h^{n+1},\nabla \cdot v_h)_{\Omega} - (p_h^{n+1},v_h \cdot n)_{\partial
  \Omega} + R_C(v_h),\, \forall v_h \in V_h.
\]
\end{enumerate}
\caption{Splitting scheme for the inviscid flow equations using stabilised FEM}\label{alg:split}
\end{algorithm}
To obtain a scheme that is stable in both the inviscid and the viscous regime a modification is proposed in algorithm \ref{alg:split2}. Here the viscous term is handled explicitly in the pressure Poisson equation, but implicitly in the velocity transport step. The stability of algorithm \ref{alg:split2} is explored in the numerical section and the method appears robust in all regimes under similar CFL conditions as above. The analysis of algorithm \ref{alg:split2} is beyond the scope of the present work.
\begin{algorithm}
 \flushleft
Given $u_h^{n}$ and $u_h^{n-1}$ we compute
\begin{enumerate}
\item Let $u^* = 1.5 u_h^{n} - 0.5 u_h^{n-1}$ and find $D_h u_h^* \in V_h$ such that
\[
(D_h u_h^*, v_h)_{\Omega} = (C_h u_h^*, v_h)_{\Omega} + (A_h u_h^*, v_h)_{\Omega},
\]
for all $v_h \in V_h$.  
\item Let
$
R_D(v) = -(D_h u_h^*, v)_{\Omega}+(\tau^{-1} u_h^{n},v)_{\Omega}.
$
for all $v \in L^2(\Omega)$.
\item Solve for the pressure: Find $p_h^{n+1} \in Q_h$ such that
\[
(\nabla p_h^{n+1},\nabla q_h)_{\Omega} = R_D(\nabla q_h), \, \forall q_h \in Q_h.
\]
\item Solve for $u_h^{n+1} \in V_h$ such that $\forall v_h \in V_h$
\[
(\tau^{-1} u_h^{n+1},v_h)_{\Omega} + (A_h \bar u_h^{n+1}, v_h)_{\Omega}
= (p_h^{n+1},\nabla \cdot v_h)_{\Omega} - (p_h^{n+1},v_h \cdot n)_{\partial
  \Omega} + R_C(v_h)
\]
where $R_C$ is defined in \eqref{eq:RC}.
\end{enumerate}
\caption{Splitting scheme for the viscous flow equations using stabilised FEM}\label{alg:split2}
\end{algorithm}
\section{Numerical examples}\label{sec:numeric}
We consider the Navier-Stokes equations in two spatial dimensions with unit density. Below we will report results on the Taylor--Green vortex and 
we present computations on the Kelvin-Helmholtz shear layer and a modified version of the benchmark of flow around a cylinder at Reynolds number 20. 
The main observation are that the extrapolated/split scheme is as robust as the monolithic scheme under the CFL condition and has the same accuracy, as predicted by theory. However in the Kelvin-Helmholtz case where transition is important, the IMEX methods go into transition slightly earlier than the solutions of the monolithic scheme showing the influence of the additional dissipation of the splitting.

\subsection{Convergence Study: Taylor--Green vortex} Consider a 2D Taylor-Green vortex moving at the phase speed of one in the $x$ direction. It was proposed by Minion and Saye \cite{Minion:2018:Saye} and had the exact solution.
\begin{align} \label{exact_sol}
    u_1^{ex}(x,y,t):&=1+\sin(2\pi (x-t))\cos(2\pi y)\exp(-8\pi^2 \mu t),  \nonumber\\
    u_2^{ex}(x,y,t):&=-\cos(2\pi (x-t))\sin(2\pi y)\exp(-8\pi^2 \mu t), \nonumber\\
    p^{ex}(x,y,t):&=\frac{1}{4}(\cos(4\pi (x-t))+\cos(4\pi y))\exp(-16\pi^2 \mu t).
\end{align}
Consider the model problem \ref{eq:convdiff} with coefficients  $\mu= 3.571\cdot 10^{-6}$ in the domain $\Omega=[0,1]\times [0,1]$. 
The right-hand side function $f$ and the Dirichlet boundary conditions are chosen according to the exact solution (\ref{exact_sol}). 
The numerical simulations are performed using the Crank-Nicolson time-stepping combined with implicit-explicit treatment of the velocity \eqref{eq:momentum_CN}--\eqref{eq:mass_CN}.
The viscosity is chosen to correspond to a high Reynolds number. The stabilization parameter for the discrete variational formulation (\ref{eq:momentum_CN}) is chosen as  $\gamma_u=0.001$. A stabilization parameter was added to the pressure stabilization \eqref{pres_stab} and set to $\gamma_p = 0.001$ in the numerical examples.
For piecewise affine $P_1$ approximations, the hyperbolic CFL is used, i.e., $\tau= Co\, h \;(\beta_{\infty} = 1)$, $Co = 0.05$  and for piecewise quadratic $P_2$ approximations, the 4/3-CFL, $\tau= Co_{4/3} h^{4/3} $ is used with $Co_{4/3} = 0.025$. The Courant numbers were not optimized for stability but were taken similarly to those in \cite{BEF17}.
Set $(\tau,h) \in \{(0.05 \cdot 0.2/2^i, 0.2/2^i)\}^4_{i=1}$ for $P_1$ approximations and $\tau=0.025h^{4/3}$, $h \in \{0.2/2^i\}^4_{i=1}$ for $P_2$ approximations. 
Next, the convergence plots for the \eqref{eq:momentum_CN}--\eqref{eq:mass_CN} and (\ref{alg:split}) schemes are illustrated in Figure \ref{time_error_1111}. The errors are measured with respect to  $\norm{u(T)- u_h^N}$ and $\norm{p(T)- p_h^N}$ at $T=1$.
The values on the $x$-axis represent space step sizes, while the values on the $y$-axis represent $L^2$-errors in velocity and pressure. Dotted (resp. dashed) lines represented the CN- split (resp. CN-imex) scheme. Triangle markers describe the $L^2$-error of the velocity, and squares markers represent pressure. The slopes are depicted by solid lines markers. We can observe a convergence $\mathcal{O}(h^{2.5})$ in $ {L}^{2}$-norm for the velocity and pressure with $P_2$ approximation, and a convergence $\mathcal{O}(h^{2})$ in the ${L}^{2}$-errors of the velocity and pressure with $P_1$ approximation.
 As per expectations, the computational results agree with the theoretical predictions.

\begin{figure}[h!]
\centering 
{\includegraphics[width=0.4\linewidth, angle=0]{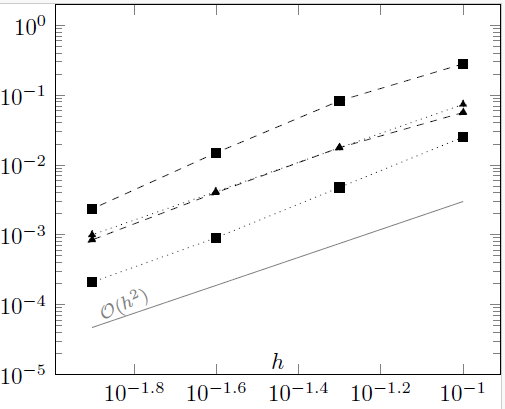}} \hspace{.35 cm}
 \includegraphics[width=0.4\linewidth, angle=0]{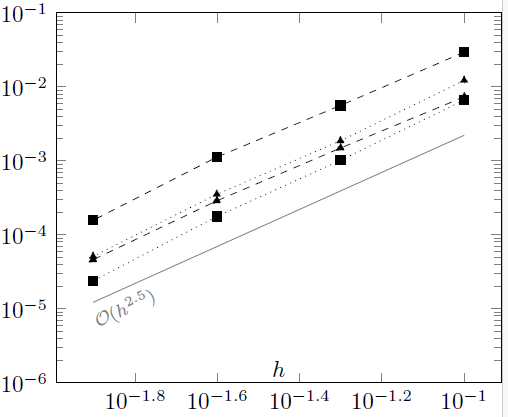} 
\caption{ Convergence plot of the Crank-Nicolson discretization applied to the problem defined by the exact solution (\ref{exact_sol}). Dashed  (resp. Dotted) lines represented the CN- IMEX (resp. CN-Split) scheme. Triangle markers describe the $ L^2$-error of the velocity, and squares markers represent pressure. Solid lines mark depicts the slopes.
left:  $P_1$ approximation of velocity and pressure; right: $P_2$ approximation of velocity and pressure.  
  }\label{time_error_1111}
\end{figure}


	\subsection{Navier--Stokes' mixing layer at $\Re=10^4$}
	We propose a qualitative study of a Kelvin--Helmholtz shear
	layer. This test was proposed in \cite{LSLC88} as a model problem for 2D
	turbulence and was shown to have the characteristic cubic decay of the power spectrum. 
	A schematic illustration of the problem setup is presented in Figure \ref{fig:KH_config}. The computational domain $\Omega$ is the unit square, the physical parameters are set as $\boldsymbol{u}_\infty = 1$, $\sigma_0 = \frac{1}{28}$, and
	$\mu = 3.571\cdot 10^{-6}$ leading to a Reynolds number based on the layer width of
	$Re= 10000$ (see Figure. \ref{fig:KH_config}). The initial data is taken as the sum of a smooth, but strongly variying velocity field and a small perturbation
	\begin{align*}
		u_0=\begin{pmatrix}
			1 \\
			0 
		\end{pmatrix}u_{\infty} \tanh\Bigl(\frac{2y-1}{\sigma_0}\Bigr)+\begin{pmatrix}
			\partial_y{\xi} \\
			-\partial_x{\xi} 
		\end{pmatrix}.
	\end{align*}
	The stream function $\xi$ specifies the form of the perturbation. In our case four vortices given by 
	\[\xi=c u_{\infty}\exp\Bigl(-\frac{(y-0.5)^2}{\sigma^2_0}\Bigr)\cos(\theta x)\]
	where $c$ is a parameter giving the strength of the perturbation, here we have chosen $c = 0.001$, and $\theta=8\pi$.
	We present snapshots at the non-dimensional times $t=80,120, 140$ using the polynomial order $k=2$ on a $80 \times 80$ mesh and the Courant number $Co_{4/3}=0.025$ resulting in the time step $7.252 \cdot
	10^{-5}$, i.e., $\tau=Co_{4/3}h^{4/3}$ and $h=0.0125$. 
	Observe that the splitting method is valid here also for the Navier-Stokes' equations; since the mesh Reynolds number is so large, the viscous term can be handled explicitly. The objective is to determine if the use of second-order time discretization remains stable while delivering the expected enhanced accuracy, as well as to determine how splitting influences the approximation accuracy.
	We compare with the monolithic solutions using implicit continuous interior
	penalty stabilization for both velocities and pressures with  $k=2$ and a $80\times 80$ mesh and the time step $\tau=3.125 \cdot 10^{-4}$. Figure \ref{fig:KH_split_vs_monol} shows a comparison between the solutions obtained using the Crank-Nicolson discretization using either the imex scheme (a) or the split method (b) on a mesh with $80 \times 80$ elements using piecewise quadratic approximation. We report a series of snapshots at the non-dimensional times $t=80, 120, 140.$ In Figure \ref{fig:KH_split_vs_monol} (c), we present the reference solution obtained using CN-time discretization and a monolithic solver using the same space discretization and at the same time levels. Figure \ref{fig:KH_split_vs_monol} (a)--(b) shows that the solutions of both imex and split schemes are of similar quality. The two vortices are merging in the first snapshot at $t=80.$ The solution obtained using a monolithic solver is entering the transition phase in the first snapshot. A possible explanation of this is that the imex and split schemes are more dissipative since it is known that excessive dissipation tends to speed up the transition sequence. 
	In Figure \ref{fig:KH_split_vs_monol_1}, we studied the time evolution of the kinetic energy, physical dissipation and artificial dissipation for different schemes. 
	The evolution of kinetic energy is given in the 
	of Figure \ref{fig:KH_split_vs_monol_1}(a). Here the loss of kinetic energy on  $P_2$ mesh is clear.  The transition points can also be studied in the plots
	of artificial or physical dissipation of Figures. \ref{fig:KH_split_vs_monol_1}(b) and \ref{fig:KH_split_vs_monol_1}(c).

	\begin{figure}[h!]
		\centering
		\includegraphics[width=0.55\linewidth]{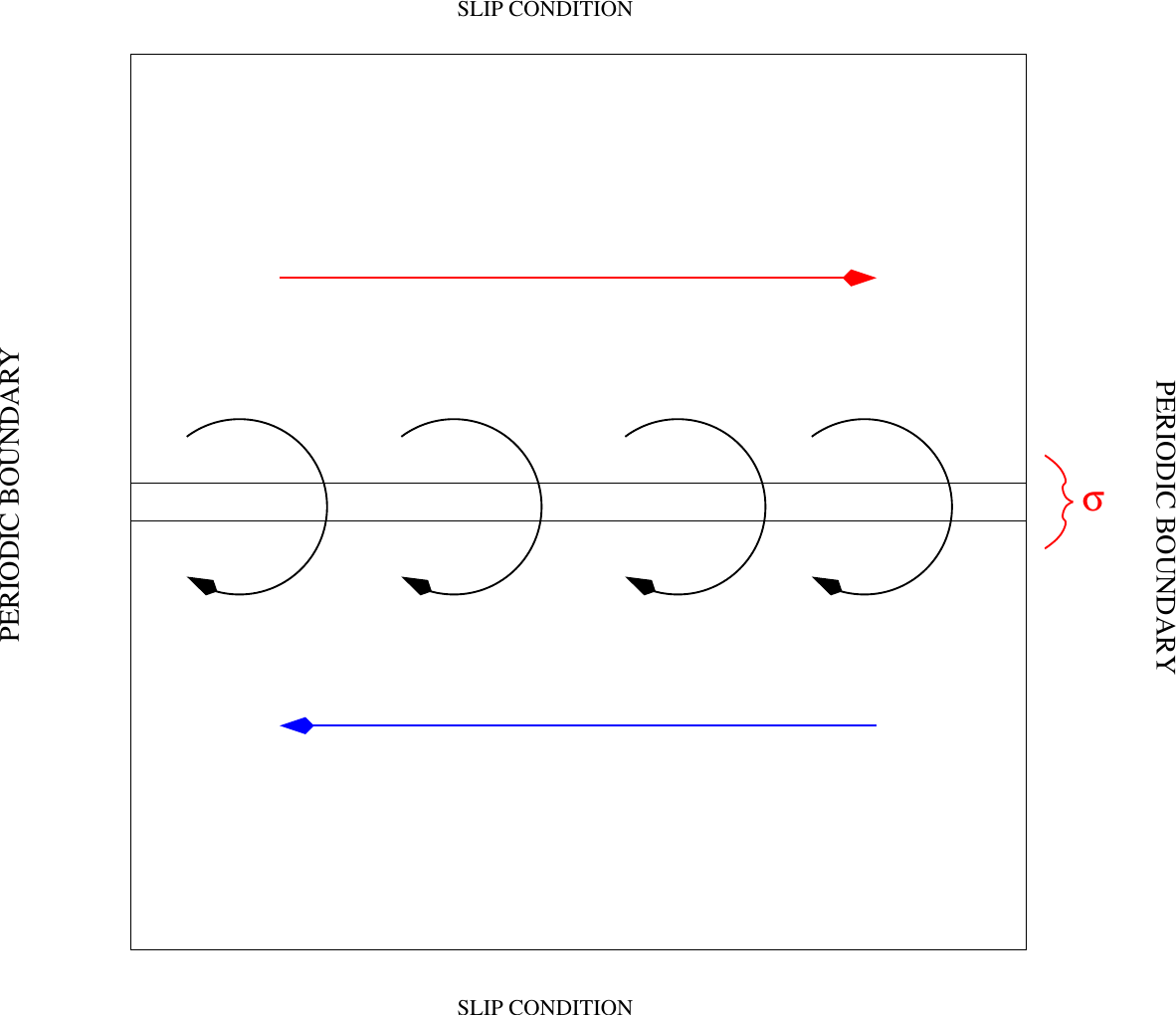}
		\caption{Computational configuration for the Kelvin--Helmholtz shear
			layer instability}\label{fig:KH_config}
	\end{figure}

	\begin{figure}[h!]
		\centering 
		{\includegraphics[width=0.41\linewidth, angle=0]{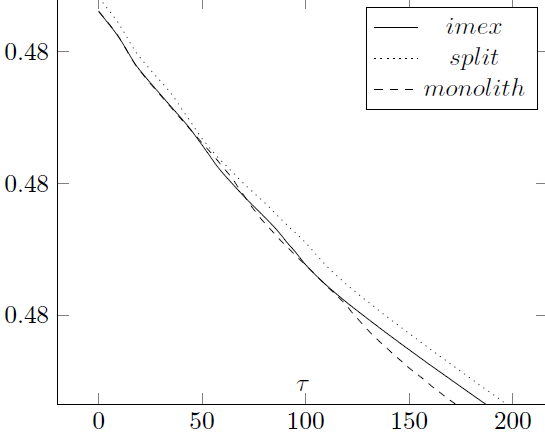} \hspace{.3 cm}
			\includegraphics[width=0.4\linewidth, angle=0]{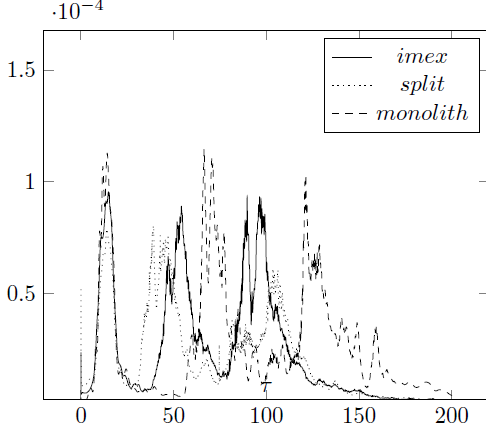} 
			\includegraphics[width=0.4\linewidth, angle=0]{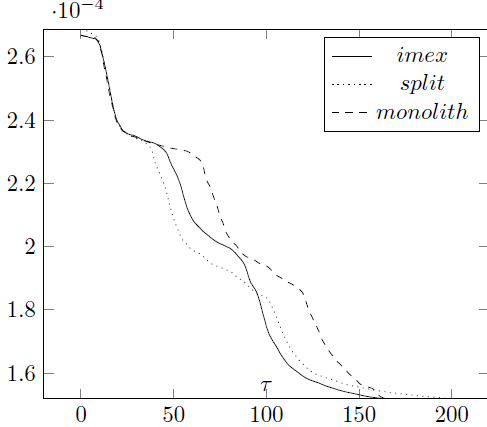}} \\
		\caption{Time evolution of (left to right) the kinetic energy, artificial dissipation and physical dissipation in computational mesh: $80 \times 80$; piecewise quadratic approximation. 
		}\label{fig:KH_split_vs_monol_1}
	\end{figure}

\begin{figure}[h!]
\centering 
\subfigure[CN-imex, $k=2$, $80 \times
80$ mesh]{\includegraphics[width=0.3\linewidth, angle=0]{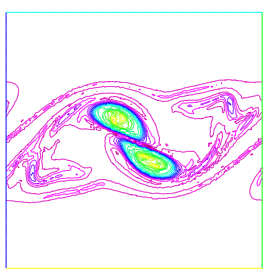}
	\includegraphics[width=0.3\linewidth, angle=0]{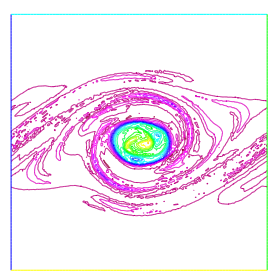}
	\includegraphics[width=0.3\linewidth, angle=0]{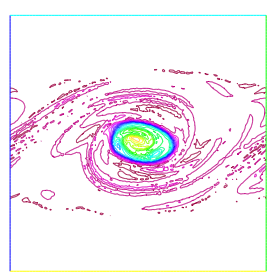}}\\
\subfigure[ CN-split, $k=2$, $80 \times
80$ mesh]{\includegraphics[width=0.29\linewidth, angle=0]{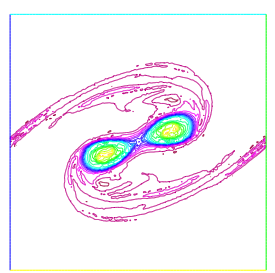}
	\includegraphics[width=0.3\linewidth, angle=0]{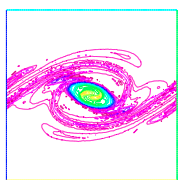}
	\includegraphics[width=0.3\linewidth, angle=0]{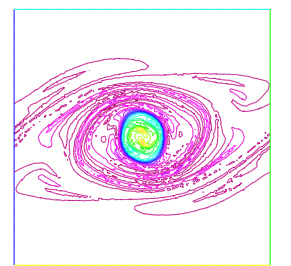}}\\
\subfigure[CN-monolithic,  $k=2$, $80 \times 80$ mesh]{\includegraphics[width=0.3\linewidth]{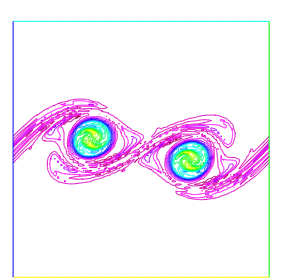}
	\includegraphics[width=0.29\linewidth]{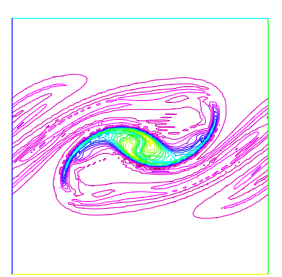}
	\includegraphics[width=0.3\linewidth]{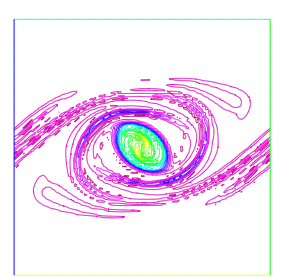}}
\caption{Comparison of  CN-imex (a), CN-split
	(b) ;  CN-monolith (c); time levels, from left to right, $t=80, 120, 140$;
	computational mesh: $80 \times 80$; piecewise quadratic approximation. 
}\label{fig:KH_split_vs_monol}
\end{figure}

\subsection{Splitting scheme with known solution $Re<1$}
Consider the Navier--Stokes equations with given exact solution.
\begin{align} \label{exact_sol_ex2}
	u_1^{ex}(x,y,t):&=2\cos(t)\sin(\pi x)\sin(\pi x)y(1-y)(1-2y),  \nonumber\\
	u_2^{ex}(x,y,t):&=(-\pi)(\cos(t))\sin(2\pi x)y^2(1-y)^2, \nonumber\\
	p^{ex}(x,y,t):&=\sin(\pi x)\cos(\pi y)\cos(t).
\end{align}
Consider the model problem \cite[eq. 1.1--1.4]{BGG24} with coefficients  $\mu= 0.1$ in the domain $\Omega=[0,1]\times [0,1]$. 
The right-hand side function $f$ and the Dirichlet boundary conditions are chosen according to the exact solution (\ref{exact_sol_ex2}). 
The numerical simulations are performed using the Crank-Nicolson splitting scheme for the viscous flow \cite[Algorithm 2]{BGG24}.
The viscosity is chosen to correspond to a low Reynolds number. The stabilization parameter for the discrete variational formulation \cite[eq. 2.7]{BGG24} is chosen as  $\gamma_u=0.001$. 
Set $(\tau,h) \in \{(0.1 \cdot 0.2/2^i, 0.2/2^i)\}^4_{i=1}$ for $P_1$ approximations and $\tau=0.025h$, $h \in \{0.2/2^i\}^4_{i=1}$ for $P_2$ approximations. 
Next, the convergence plots for the proposed scheme are illustrated in Figure \ref{low_re}. The errors are measured with respect to  $\norm{u(T)- u_h^N}$ and $\norm{p(T)- p_h^N}$ for the final time $T=1.1$.
The values on the x-axis represent space step sizes, while the values on the y-axis represent $L^2$-errors in velocity and pressure. 
Triangle markers describe the $ L^2$-error of the velocity, and squares markers represent pressure. The slopes are depicted by solid lines markers. We can observe a convergence $\mathcal{O}(h^{3})$ in $ {L}^{2}$-norm for the velocity and pressure with $P_2$ approximation, and a convergence $\mathcal{O}(h^{2})$ in the ${L}^{2}$-errors of the velocity and pressure with $P_1$ approximation.
As per expectations, the computational results agree with the theoretical predictions.

\begin{figure}[ht!]
	\centerline{\includegraphics[width=7.0cm]{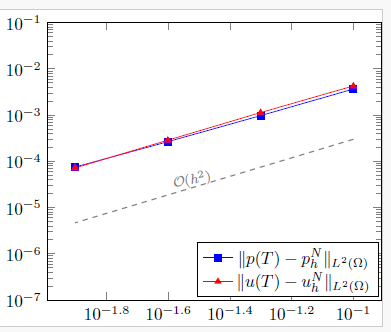}\hspace*{0cm}\includegraphics[width=7.2cm]{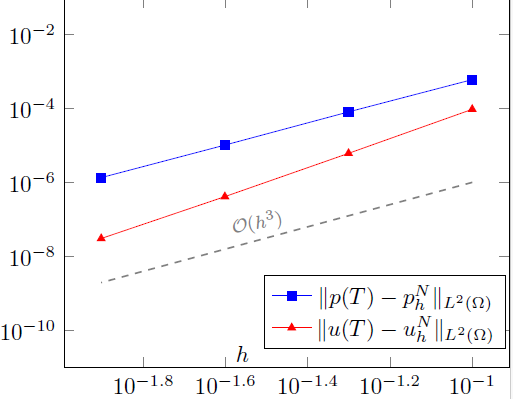}
	}
	\caption{{ Convergence plot of the Crank-Nicolson splitting scheme in the viscous flow applied to the problem defined by the exact solution (\ref{exact_sol_ex2}). Absolute-error. Left $P_1 $ Approximations.   Right $P_2 $ Approximations.  } } \label{low_re}
\end{figure}

\subsection{Navier--Stokes' flow around a cylinder at $Re=20$}

\begin{figure}[h!]
	\centering
	\includegraphics[width=1.0\linewidth]{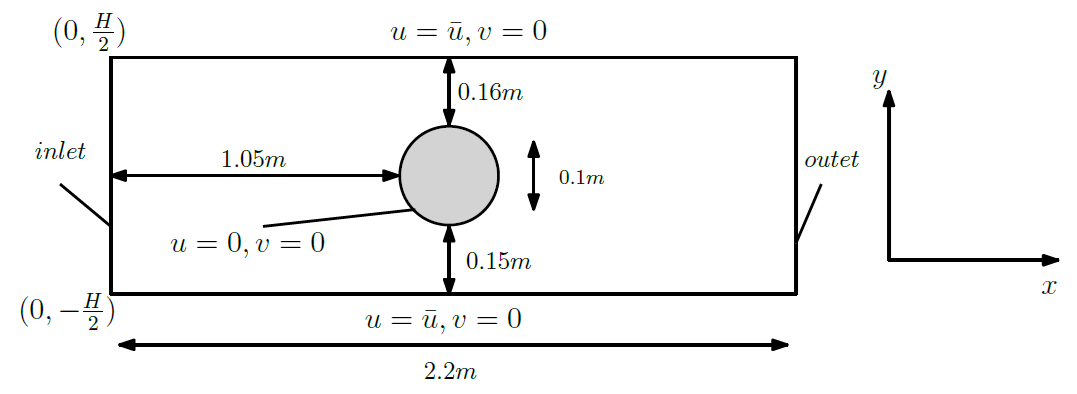}
	\caption{Geometry of the flow past cylinder domain with boundary conditions}\label{cyl_domain}
\end{figure}

\begin{figure}[h!]
	\centering
	\includegraphics[width=0.7\linewidth]{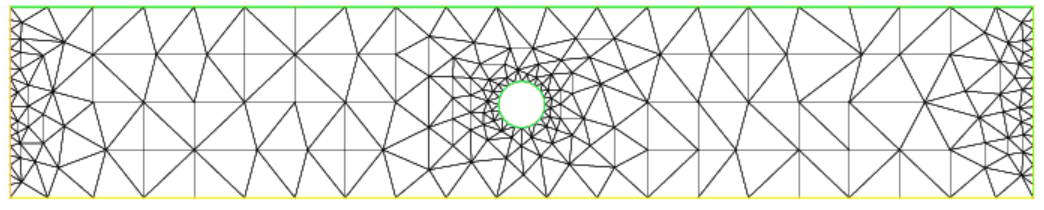}
	\caption{Cylinder initial mesh}\label{cyl_mesh}
\end{figure}
Here, we consider a modified version of the benchmark of flow around a cylinder at Reynolds number 20 proposed in \cite{schafer1996benchmark}. 
The geometry of the cylinder is presented in Figure \ref{cyl_domain}. Some definitions are introduced to specify the values which have to be computed. $H=0.41 m$ is the channel height $D=0.1 m$ is the cylinder diameter. The Reynolds number is defined as $Re= \frac{\bar{u} D}{\mu}$ with  $\bar u =0.2$.  The coefficients $\mu$ is chosen as $ 10^{-3}$.
The initial mesh is constructed with 20 elements on the cylinder, 20
elements on the horizontal boundaries and 20 elements on the vertical boundaries see \ref{cyl_mesh}. The boundary conditions are periodic on the inlet and outlet of the channel and $u=\bar{u}$ and $v=0$  velocities are set on the top and the bottom of the channel. On the cylinder homogeneous Dirichlet boujndary conditions are set. The
finer meshes are obtained by multiplying the number of elements on each boundary by two, four 
and eight, respectively.  To assess the quality of the cylinder computation, we compute the monolithic solution in parallel and compare the two solutions in $L^2$-norm and $H^1$-norm; since we know that the monolithic solution has optimal convergence, this comparison gives us directly the convergence order of the split scheme. The stabilization parameter for the discrete variational formulation \cite[eq. 3.2]{BGG24} is chosen as  $\gamma_u=0.001$. A stabilization parameter was added to the pressure stabilization \cite[eq. 2.7]{BGG24} in the monolith scheme and set to $\gamma_p = 0.001$ in the numerical examples.
Let $\tau= Co\, h \;(\beta_{\infty} = 1)$, $Co$ is the hyperbolic CFL.
Set $(\tau,h) \in \{(0.1 \cdot 0.1/2^i, 0.1/2^i)\}^4_{i=1}$ for $P_1$ approximations and $\tau=0.01h$, $h \in \{0.1/2^i\}^4_{i=1}$ for $P_2$ approximations. 
Next, the convergence plots for the scheme \cite[Algorithm 2]{BGG24}  are illustrated in Figure \ref{cong_cyl}.
The values on the x-axis represent space step sizes, while the values on the y-axis represent errors in velocity.  Triangle markers describe the $L^2$-error, and Circle markers represent the $H^1$-error of the velocity. The slopes are depicted by solid lines markers.  The errors are measured at final time $T=1$.
We observe optimal convergence for the velocities using the piecewise affine approximation, i.e. $O(h)$ in the $H^1$-norm and $O(h^2)$ in the $L^2$-norm.
In the case of piecewise quadratic approximation, the convergences are suboptimal with $O(h^{\frac{1}{2}})$. This reduction in convergence rate is likely due to the fact that, as discussed in \cite{BGG23}, the splitting scheme introduces implicitly a splitting error which matches the size of the stabilizing term in the high Reynolds regime. To keep optimal convergence in the low Reynolds regime the stabilization of the pressure should be relaxed by an order that matches the missing $h^{\frac{1}{2}}$, but that is impossible here since the stabilization term is never constructed. Another source of inconsistency incompatible with optimal convergence using quadratic elements is the piecewise affine approximation of the curved cylinder boundary.

\begin{figure}[ht!]
	\centerline{\includegraphics[width=7.0cm]{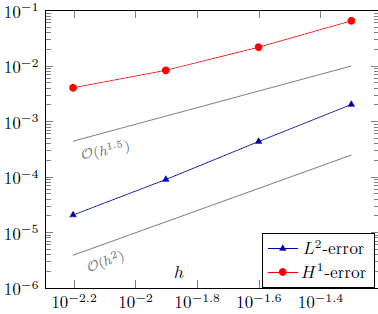}\hspace*{0cm}\includegraphics[width=7.0cm]{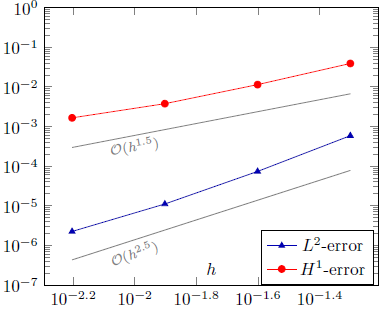}
	}
	\caption{{Absolute-error. Left $P_1 $ approximation.   Right $P_2 $ approximation. Convergence results for the splitting scheme. } } \label{cong_cyl}
\end{figure}

\section*{Acknowledgement}
The first and second authors acknowledge funding from EPSRC grant EP/T033126/1.

\bibliographystyle{plain}
\bibliography{references}

\begin{thebibliography}{10}

\bibitem{BB06}
M.~Braack and E.~Burman.
\newblock Local projection stabilization for the {O}seen problem and its
  interpretation as a variational multiscale method.
\newblock {\em SIAM J. Numer. Anal.}, 43(6):2544--2566, 2006.

\bibitem{BE07}
Erik Burman and Alexandre Ern.
\newblock Continuous interior penalty {$hp$}-finite element methods for
  advection and advection-diffusion equations.
\newblock {\em Math. Comp.}, 76(259):1119--1140, 2007.

\bibitem{BEF17}
Erik Burman, Alexandre Ern, and Miguel~A. Fern\'{a}ndez.
\newblock Fractional-step methods and finite elements with symmetric
  stabilization for the transient {O}seen problem.
\newblock {\em ESAIM Math. Model. Numer. Anal.}, 51(2):487--507, 2017.

\bibitem{BFH006}
Erik Burman, Miguel~A. Fern\'{a}ndez, and Peter Hansbo.
\newblock Continuous interior penalty finite element method for {O}seen's
  equations.
\newblock {\em SIAM J. Numer. Anal.}, 44(3):1248--1274, 2006.

\bibitem{BGG24}
Erik Burman, Deepika Garg, and Johnny Guzman.
\newblock The supplementary material for the paper \\ "{I}mplicit-{E}xplicit
  {C}rank-{N}icolson scheme for {O}seen's equation at high {R}eynolds number".
\newblock {\em submitted}.

\bibitem{BGG23}
Erik Burman, Deepika Garg, and Johnny Guzman.
\newblock Implicit-{E}xplicit {T}ime {D}iscretization for {O}seen's {E}quation
  at {H}igh {R}eynolds {N}umber with {A}pplication to {F}ractional {S}tep
  {M}ethods.
\newblock {\em SIAM J. Numer. Anal.}, 61(6):2859--2886, 2023.

\bibitem{BG22}
Erik Burman and Johnny Guzm\'{a}n.
\newblock Implicit-explicit multistep formulations for finite element
  discretisations using continuous interior penalty.
\newblock {\em ESAIM Math. Model. Numer. Anal.}, 56(1):349--383, 2022.

\bibitem{BH04}
Erik Burman and Peter Hansbo.
\newblock Edge stabilization for {G}alerkin approximations of
  convection-diffusion-reaction problems.
\newblock {\em Comput. Methods Appl. Mech. Engrg.}, 193(15-16):1437--1453,
  2004.

\bibitem{CCRL17}
F.~Capuano, G.~Coppola, L.~R\'{a}ndez, and L.~de~Luca.
\newblock Explicit {R}unge-{K}utta schemes for incompressible flow with
  improved energy-conservation properties.
\newblock {\em J. Comput. Phys.}, 328:86--94, 2017.

\bibitem{Cod08}
Ramon Codina.
\newblock Analysis of a stabilized finite element approximation of the {O}seen
  equations using orthogonal subscales.
\newblock {\em Appl. Numer. Math.}, 58(3):264--283, 2008.

\bibitem{DD76}
Jim Douglas, Jr. and Todd Dupont.
\newblock Interior penalty procedures for elliptic and parabolic {G}alerkin
  methods.
\newblock In {\em Computing methods in applied sciences ({S}econd {I}nternat.
  {S}ympos., {V}ersailles, 1975)}, pages 207--216. Lecture Notes in Phys., Vol.
  58. 1976.

\bibitem{ESW}
Howard~C. Elman, David~J. Silvester, and Andrew~J. Wathen.
\newblock {\em Finite elements and fast iterative solvers: with applications in
  incompressible fluid dynamics}.
\newblock Numerical Mathematics and Scientific Computation. Oxford University
  Press, Oxford, second edition, 2014.

\bibitem{FWK18}
Niklas Fehn, Wolfgang~A. Wall, and Martin Kronbichler.
\newblock Robust and efficient discontinuous {G}alerkin methods for
  under-resolved turbulent incompressible flows.
\newblock {\em J. Comput. Phys.}, 372:667--693, 2018.

\bibitem{FF92}
Leopoldo~P. Franca and S\'{e}rgio~L. Frey.
\newblock Stabilized finite element methods. {II}. {T}he incompressible
  {N}avier-{S}tokes equations.
\newblock {\em Comput. Methods Appl. Mech. Engrg.}, 99(2-3):209--233, 1992.

\bibitem{he2007stability}
Yinnian He and Weiwei Sun.
\newblock Stability and convergence of the
  {C}rank--{N}icolson/{A}dams--{B}ashforth scheme for the time-dependent
  {N}avier--{S}tokes equations.
\newblock {\em SIAM Journal on Numerical Analysis}, 45(2):837--869, 2007.

\bibitem{johnston2004accurate}
Hans Johnston and Jian-Guo Liu.
\newblock Accurate, stable and efficient {N}avier--{S}tokes solvers based on
  explicit treatment of the pressure term.
\newblock {\em Journal of Computational Physics}, 199(1):221--259, 2004.

\bibitem{KMR14}
Songul Kaya, Carolina~C. Manica, and Leo~G. Rebholz.
\newblock On {C}rank-{N}icolson {A}dams-{B}ashforth timestepping for
  approximate deconvolution models in two dimensions.
\newblock {\em Appl. Math. Comput.}, 246:23--38, 2014.

\bibitem{kim1985application}
John Kim and Parviz Moin.
\newblock Application of a fractional-step method to incompressible
  {N}avier-{S}tokes equations.
\newblock {\em Journal of computational physics}, 59(2):308--323, 1985.

\bibitem{LMS23}
Philip~L. Lederer, Xaver Mooslechner, and Joachim Sch\"{o}berl.
\newblock High-order projection-based upwind method for implicit large eddy
  simulation.
\newblock {\em J. Comput. Phys.}, 493:Paper No. 112492, 2023.

\bibitem{LSLC88}
M.~Lesieur, C.~Staquet, P.~Le~Roy, and P.~Comte.
\newblock The mixing layer and its coherence examined from the point of view of
  two-dimensional turbulence.
\newblock {\em J. Fluid Mech.}, 192:511--534, 1988.

\bibitem{marion1998navier}
Martine Marion and Roger Temam.
\newblock {N}avier-{S}tokes equations: Theory and approximation.
\newblock {\em Handbook of numerical analysis}, 6:503--689, 1998.

\bibitem{Minion:2018:Saye}
M.~L. Minion and R.~I. Saye.
\newblock Higher-order temporal integration for the incompressible
  {N}avier-{S}tokes equations in bounded domains.
\newblock {\em J. Comput. Phys.}, 375:797--822, 2018.

\bibitem{MRTT16}
Monica Morales~Hernandez, Leo~G. Rebholz, Cristina Tone, and Florentina Tone.
\newblock Stability of the {C}rank-{N}icolson-{A}dams-{B}ashforth scheme for
  the 2{D} {L}eray-{A}lpha model.
\newblock {\em Numer. Methods Partial Differential Equations},
  32(4):1155--1183, 2016.

\bibitem{MCBS22}
Rodrigo~C. Moura, Andrea Cassinelli, Andr\'{e} F.~C. da~Silva, Erik Burman, and
  Spencer~J. Sherwin.
\newblock Gradient jump penalty stabilisation of spectral/{$hp$} element
  discretisation for under-resolved turbulence simulations.
\newblock {\em Comput. Methods Appl. Mech. Engrg.}, 388:Paper No. 114200, 29,
  2022.

\bibitem{schafer1996benchmark}
Michael Sch{\"a}fer, Stefan Turek, Franz Durst, Egon Krause, and Rolf
  Rannacher.
\newblock {\em Benchmark computations of laminar flow around a cylinder}.
\newblock Springer, 1996.

\bibitem{Thom97}
Vidar Thom\'{e}e.
\newblock {\em Galerkin finite element methods for parabolic problems},
  volume~25 of {\em Springer Series in Computational Mathematics}.
\newblock Springer-Verlag, Berlin, 1997.

\bibitem{tone2004error}
Florentina Tone.
\newblock Error analysis for a second order scheme for the {N}avier--{S}tokes
  equations.
\newblock {\em Applied numerical mathematics}, 50(1):93--119, 2004.

\bibitem{ZJH18}
Tong Zhang, JiaoJiao Jin, and YuGao HuangFu.
\newblock The {C}rank-{N}icolson/{A}dams-{B}ashforth scheme for the {B}urgers
  equation with {$H^2$} and {$H^1$} initial data.
\newblock {\em Appl. Numer. Math.}, 125:103--142, 2018.

\end{thebibliography}

	

\end{document}